\documentclass[12pt]{article}

\usepackage{mathtools}
\usepackage{amsfonts}
\usepackage{amsthm}
\usepackage{amssymb}

\usepackage{graphicx}

\usepackage{lie-ieee-macros}

\usepackage{caption}
\usepackage{subcaption}

\usepackage{url}

\begin{document}
\title{Local observers on linear Lie groups with linear estimation error
  dynamics}
\author{Mikhail~Koldychev~and~Christopher~Nielsen
\thanks{Supported by the Natural Sciences and Engineering
Research Council of Canada (NSERC).}
\thanks{M. Koldychev and C. Nielsen are with the Department of Electrical and
Computer Engineering, University of Waterloo, 200 University Avenue West,
Waterloo, Ontario, Canada, N2L 3G1. {\tt\small \{mkoldych,
cnielsen\}@uwaterloo.ca}}
}

\markboth{}%
{Koldychev, Nielsen : Local observers on linear Lie groups with linear estimation error
  dynamics}

\maketitle

\begin{abstract}
  This paper proposes local exponential observers for systems on
  linear Lie groups. We study two different classes of systems. In the
  first class, the full state of the system evolves on a linear Lie
  group and is available for measurement. In the second class, only
  part of the system's state evolves on a linear Lie group and this
  portion of the state is available for measurement. In each case, we
  propose two different observer designs. We show that, depending on
  the observer chosen, local exponential stability of one of the two
  observation error dynamics, left or right invariant error dynamics,
  is obtained. For the first class of systems these results are
  developed by showing that the estimation error dynamics are
  differentially equivalent to a stable linear differential equation
  on a vector space. For the second class of system, the estimation
  error dynamics are almost linear. We illustrate these observer
  designs on an attitude estimation problem.
\end{abstract}

\section{Introduction}
\label{sec:intro}

Observers for systems on Lie groups is an active area of
research~\cite{crassidis2007survey}. Interest in this research has
been partially motivated by the problem of controlling mobile robots,
and in particular, unmanned aerial vehicles (UAVs). Precise control of
these systems requires accurate estimates of the orientation of a
rigid-body using low cost on-board sensors
~\cite{metni2006attitude}. Small autonomous robots usually undergo
significant vibration and other disturbances, while being restricted
to carrying only a basic light-weight sensor package. For this reason,
high-frequency noise is often present in the sensor measurements of
these robots.  Nonlinear observers for systems on Lie groups are
useful because, in certain cases, they can be used to filter out the
sensor noise.

Recent work on full-state observers for systems on $\SO{3, \Real}$,
describing rigid-body rotational kinematics, was done
in~\cite{MahHamPfl:05}, ~\cite{MahHamPfl:08}. The algorithms in
~\cite{MahHamPfl:05}, ~\cite{MahHamPfl:08} rely on a projection of the
measurement error from the Lie group to its Lie algebra. The projected
vector in the Lie algebra is then used to drive the observer to
converge to the system trajectory. While this projection based
approach does not work for systems on the general linear Lie group,
$\GL{n,\Real}$, the work in~\cite{SarSep09} may contain ideas to
extend these projection based observers to the general linear group.

For systems on $\SO{3}$ with partial state measurements, the
paper~\cite{Salcudean:91} proposes an observer that uses measurements
of the orientation and of the torque to estimate the angular velocity
of the rigid-body. The papers ~\cite{batista2012ges},
~\cite{khosravian2010globally} propose globally exponentially
convergent observers using partial state measurements. The work in
~\cite{grip2012attitude} also uses partial state measurements in their
observers. The paper~\cite{FirNam:11} analyses the effect of noise on
an attitude estimation observer. The authors of~\cite{SchEbeAlg12}
propose observers for $\SO{n, \Real}$.

For systems on $\SE{3, \Real}$, describing rigid-body pose, full-state
observers were proposed in ~\cite{BalMahTruHam:07},
~\cite{HuaZamTru:11}, ~\cite{vasconcelos2010nonlinear}.  For systems
on $\SL{3, \Real}$, describing a homography transformation,
partial-state observers were proposed in~\cite{MalHamMahMor:09}.

In this paper, we consider left-invariant systems on the general
linear Lie group, i.e., the group of all invertible, real $n \times n$
matrices.  The output of the system is taken to be that portion of the
state evolving on the linear Lie group. We first consider the case in
which the entire state evolves on the Lie group. We call these Lie
group full-state observers. Then we consider the case where the states
evolving on the Lie group are only a subset of the systems entire
state. We call these Lie group partial-state observers.

A recent breakthrough in observer design on the general linear Lie
group was achieved in~\cite{LagTruMah:10}, where exponentially
converging observers are proposed for left-invariant and
right-invariant systems on arbitrary finite dimensional, connected Lie
groups. The proposed exponential observer uses gradient-like driving
terms, derived from cost functions of the Lie group measurement
error. In this paper we propose an alternative to gradient-like
observers. Our observers are noteworthy because they yield linear
estimation error dynamics.  A weakness of our result is that we only
prove local exponential stability.

A vector field on a connected Lie group is said to be linear if its
flow is a one-parameter subgroup~\cite{Jou:09}.  Recent work on linear
vector fields and linear systems on Lie groups was done
in~\cite{Jou:09} and~\cite{Jou:11}.  The estimation error dynamics of
our Lie group full-state observers are shown to be linear vector
fields on the Lie group. We also show that the estimation error
dynamics are differentially equivalent, by means of the logarithm map,
to a linear vector field on the Lie algebra.

The research in~\cite{BonMarRou:08},~\cite{BonMarRou:09} considers a
more general class of systems than the left-invariant systems
considered in this paper.  These articles consider systems that evolve
on a vector space, but are such that a certain Lie group action leaves
the system equations unchanged. This shows that, if the plant is
invariant under the action of some Lie group, then part of the states
of the plant can be redefined as evolving on this Lie group, at least
locally.

The main contributions of this paper include the exponential
full-state observer in Section~\ref{sec:LFSO_observer} for
left-invariant systems on the general linear Lie group. Our observer
yields linear estimation error dynamics, distinguishing it from other
observers in the literature. In Section~\ref{sec:LFSO_observer} we
propose an exponential partial-state observer, for a class of systems
that is a generalization of the left-invariant systems considered in
~\ref{sec:LFSO_observer}. This class of systems has only a proper
subset of its states evolving on the Lie group, while the rest of the
states evolve on the Lie algebra. The results in Section~\ref{sec:ODE}
show equivalence properties for certain differential equations on
$\GL{n,\Real}$. The effectiveness of the proposed observers is
illustrated via simulation in Section~\ref{sec:examples}.

\section{Motivation}
\label{sec:motivation}
Attitude control of a rigid body, like a UAV, is simplified if
accurate estimates of its orientation are available. The attitude of a
rigid body in Euclidean three-dimensional space is specified by the
relative orientation between a coordinate frame fixed to the body and
an inertial coordinate system. The orientation between the frames is
described by means of a rotation matrix. If we denote the body frame
by $\mathcal{F}_{b}$ and the inertial frame by $\mathcal{F}_a$ then
the rotation matrix $R_{ab}$ relating $\mathcal{F}_{b}$ to
$\mathcal{F}_a$ is an element of the special orthogonal group $\SO{3,
  \Real} \subset \Real^{3 \times 3}$.


If we assume that the orientation of the body frame $\mathcal{F}_b$
with respect to the inertial frame $\mathcal{F}_a$ varies in time,
then the entries of $R_{ab}$ become functions of time. By
differentiating the identities $R_{ab}R^\top_{ab} = I$ and
$R^\top_{ab}R_{ab} = I$ one can show that $\dot{R}_{ab}R^\top_{ab}$
and $R^\top_{ab}\dot{R}_{ab}$ are skew-symmetric. Define
$\Skew{(\omega^a_{ab})} \coloneqq \dot{R}_{ab}R^\top_{ab}$ where
$\Skew : \Real^3 \to \Real^{3 \times 3}$ takes any vector in $v =
\col(v_1, v_2, v_3) \in \Real^3$ and maps it to a skew-symmetric
matrix
\[
\Skew{(v)} = \left[\begin{array}{ccc}0 & -v_3 & v_2\\v_3 & 0 &
    -v_1\\-v_2 & v_1 & 0\end{array}\right].
\]
The vector $\omega^a_{ab}$ is called the instantaneous inertial
angular velocity. It represents the instantaneous angular velocity of
the rigid body as seen from $\mathcal{F}_a$. Similarly, define the
instantaneous body angular velocity $\omega^b_{ab}$ via
$\Skew{(\omega^b_{ab})} \coloneqq R^\top_{ab}\dot{R}_{ab}$. The matrix
$R_{ab}(t)$ is therefore the solution of either of the differential
equations
\[
\dot{R}_{ab} = R_{ab}\Skew{(\omega^a_{ab})}, \qquad \dot{R}_{ab} =
\Skew{(\omega^b_{ab})}R_{ab}.
\]
These equations are kinematic since they do not involve forces or
torques. In this paper we consider a class of systems that includes
the first of these two differential equations, i.e., we consider
kinematic systems of the form\footnote{The results in this paper also
  apply to equations of the form $\dot{R} = uR$. We do not explicitly
  show these derivations and proofs because they can be obtained, {\em
    mutatis mutandis}, from the results presented.}
\begin{equation}
\begin{aligned}
\dot{R} &= R u \\
Y &= R , 
\label{eq:system_so3}
\end{aligned}
\end{equation}
where $Y \in \SO{3, \Real}$ is a direct measurement of $R$.  This
equation is relevant to attitude control problems because $Y$ and $u$
can be measured or estimated using low-cost sensors, mounted on the
body frame $\mathcal{F}_b$. In fact, $u$ can be measured using angular
rate gyroscopes, while the rotation matrix $Y$ can be calculated from
sensor measurements, made by an accelerometer and magnetometer pair.

Given these measurements, in Section~\ref{sec:proposed} we design an
observer to estimate the state of~\eqref{eq:system_so3}. Our observer
enjoys the property that if the estimate of $R$ isn't too erroneous
initially, then the estimate will converge to $R$ exponentially and
the rate of convergence can be easily tuned using a single
parameter. At first glance it may seem unnecessary to design an
observer to estimate the state of a system if that state is already
measured. We are motivated to study this problem for two
reasons. First such an observer provides noise filtering and can be
used as a sensor fusion algorithm~\cite{MahHamPfl:05},
~\cite{MahHamPfl:08}. The second reason is that by looking at this
simple case we hope to gain insight into the state estimation problem
when the output is not equal to the entire state.

In the latter case, consider a dynamic model of~\eqref{eq:system_so3}
\begin{equation}
\begin{aligned}
\dot{R} &= R \omega \\
\dot{\omega} &= u \\
Y &= R
\label{eq:dynamic_so3}
\end{aligned}
\end{equation}
where $\omega$ and $u$ are skew-symmetric matrices. Here $Y = R$ is
directly measured as well as the angular acceleration $u$.  In
Section~\ref{sec:proposed} we propose observers that fuse these two
sensor measurements to obtain an estimate of angular velocity,
$\omega$, and also to filter out noise from $Y$.

\section{Notation and Preliminaries}
\label{sec:notation}
We denote by $\Real^{+}$ the set of real numbers, equipped with the
additive group structure. Let $\Real_{\leq 0} \coloneqq (-\infty ,
0]$. The empty set is $\varnothing$. If $z \in \Complex$, then
$\RE{(z)}$ and $\IM{(z)}$ denote its real and imaginary parts. We
denote by $\Real^{n \times n}$ the set of $n \times n$ matrices with
real entries. If $X \in \Real^{n \times n}$ then $\sigma(X)$ denotes
the eigenvalues of $X$. If $x \in \Real^n$, then $x_i$ refers to the
$i$th component of $x$.  If $X \in \Real^{n \times n}$, then $X_{ij}$
refers to the $(i,j)$th element of $X$. The symbols $I_n$ and $0_n$
denote the $n\times n$ identity matrix and $n \times n$ zero matrix
respectively. If $A \in \Real^{n \times n}$ then $A^\top$ denotes the
transpose of $A$ and $\tr(A)$ denotes its trace, i.e., $\tr{(A)} =
\sum^n_iA_{ii}$.

We denote by $\GL{n,\Real}$ the general linear Lie group of all
invertible $n \times n$ matrices with real entries and matrix
multiplication as the group operation
\begin{equation*}
\begin{aligned}
  \GL{n, \Real} &\coloneqq \{ X \in \Real^{n \times n} : \det(X) \neq 0 \}.
\end{aligned}
\end{equation*}
We denote by $\M{n,\Real}$ the algebra of all $n \times n$ matrices
with real entries. The bilinear product that makes $\M{n,\Real}$ an
algebra is the matrix commutator, i.e., given $A, B \in \M{n,\Real}$,
the product of $A$ and $B$ is $[A, B] \coloneqq AB - BA$.  For
matrices $A \in \M{n,\Real}$ and $X \in \GL{n,\Real}$, the adjoint map
is $\Ad_X(A) \coloneqq X A X^{-1}$.

For a vector $x \in \Real^n$, $\|x\|$ denotes the Euclidean norm. For
a matrix $A \in \Real^{n \times n}$, the induced matrix norm is
\[
\| A \| \coloneqq \max\left\{\|Ax\| : x\in\Real^n, \|x\| \leq 1\right\}.
\]
Induced norms on $\Real^{n \times n}$ are
submultiplicative~\cite{vid02}, i.e., for any $A$, $B \in \Real^{n
  \times n}$, $\| A B \| \leq \|A\| \|B\|$. Given a matrix $A \in
\Real^{n \times n}$ and a real scalar $r > 0$, define the open ball
\begin{equation}
\begin{aligned}
  B(A,r) \coloneqq \left\{X \in \Real^{n \times n} : \|A - X\| < r
  \right\}.
\end{aligned}
\end{equation}
\begin{proposition}
\label{prop:inverse_def}
Let $X \in B(I_n,1)$, then the series $\sum_{k=0}^{\infty}(-1)^k
(X-I_n)^k$ converges in norm and
\begin{equation}
\begin{aligned}
X^{-1} = \sum_{k=0}^{\infty}(-1)^k (X-I_n)^k.
\label{eq:inverse_def}
\end{aligned}
\end{equation}
\end{proposition}
\begin{proof}
Let $X \in B(I_n, 1)$ and define $M \coloneqq I_n - X$ so that
$\| M \|<1$. Taking the norm of the right hand side of
\eqref{eq:inverse_def}
\begin{equation*}
\begin{aligned}
\left\| \sum_{k=0}^{\infty}(-1)^k (X-I_n)^k \right\| & = \left\| \sum_{k=0}^{\infty}
M^k \right\|\\
&\leq \sum_{k=0}^{\infty} \| M^k \| \\
&\leq \sum_{k=0}^{\infty} \| M \|^{k} \\
&= (1 - ||M||)^{-1} .
\end{aligned}
\end{equation*}
Hence the right hand side of~\eqref{eq:inverse_def} is a convergent
series. To see that its sum equals $X^{-1}$, since $X$ is a square
matrix, it is enough to left-multiply by $X$ and to check that the
result is the identity matrix. To this end
\begin{equation*}
\begin{aligned}
X \sum_{k=0}^{\infty}(-1)^k (X - I_n)^k 
&= (I_n - M) \sum_{k=0}^{\infty} M^k    \\ 
&= \sum_{k=0}^{\infty}\left( M^k - M^{k+1} \right)  \\ 
&= I_n + \sum_{k=1}^{\infty}\left( M^k \right) 
       - \sum_{k=0}^{\infty}\left( M^{k+1} \right) \\
&= I_n + M - M + M^2 - M^2 + \cdots \\  
&= I_n .
\end{aligned}
\end{equation*}
\end{proof}
A consequence of Proposition~\ref{prop:inverse_def} is that the set
$B(I_n,1)$ is contained in $\GL{n,\Real}$. 

\subsection{Linear Lie groups}
In this section, we give a brief introduction to Lie groups as
subgroups of matrices.  The main mathematical references
are~\cite{Faraut:08},~\cite{Che99}. All the results in this section
are standard.

\begin{definition}
A linear Lie group $\g$ is a closed subgroup of $\GL{n,\Real}$.
\end{definition}
A linear Lie group is not the same thing as a matrix Lie group. A
matrix Lie group is a subgroup of $\GL{n,\Real}$ (\cite[Definition
5.13]{Bullo:05}), but not necessarily closed as a set. For example the
subgroup $\GL{n, \mathbb{Q}}$ of rational matrices with rational
inverses is not closed in the topology of $\Real^{n\times n}$. The
fact that linear groups are additionally restricted to be closed sets
ensures, by the Closed Subgroup Theorem~\cite[Theorem 20.10]{Lee02},
that the Lie group is an embedded submanifold of $\Real^{n \times
  n}$. Working with Lie groups that are embedded submanifolds of
$\Real^{n \times n}$, rather than abstract manifolds, allows us to
perform computations in the embedding space. In particular, this
permits the use of standard vector calculus. The idea of doing
computations in a larger and simpler embedding manifold is one of the
key ideas in~\cite{SarSep09}, see
also~\cite{absil2009optimization}. For brevity, the term Lie group is
used in place of linear Lie group throughout. In addition to
$\GL{n,\Real}$, the following linear Lie groups are used in this paper
\begin{equation*}
\begin{aligned}
  \SL{n, \Real} &\coloneqq \{ X \in \Real^{n \times n} : \det(X) = 1 \}, \\
  \SO{n, \Real} &\coloneqq \{ X \in \Real^{n \times n} : X X^{\top} =
  I_n, \det(X) = 1 \} .
\end{aligned}
\end{equation*}

\begin{definition}
Given a matrix $A \in \M{n,\Real}$, the matrix exponential 
$\exp: \M{n,\Real} \rightarrow \GL{n,\Real}$ is defined to be
\begin{equation}
\begin{aligned}
\exp(A) \coloneqq \sum_{k=0}^{\infty} \frac{A^k}{k!} .
\label{eq:exponential_def}
\end{aligned}
\end{equation}
\end{definition}
Since $\|A^k\| \leq \|A\|^k$, the series~\eqref{eq:exponential_def}
converges in norm for every matrix $A \in \Real^{n \times n}$.
\begin{definition}
  Given a linear Lie group $\g$, the Lie algebra of $\g$, denoted by
  $\Lie(\g)$, is the set
\begin{equation*}
\begin{aligned}
  \Lie(\g) \coloneqq \left\{ A \in \M{n,\Real} : \; \left(\forall \;
      t \in \Real\right) \exp(t A) \in \g \right\}.
\end{aligned}
\end{equation*}
\label{def:liealg}
\end{definition}
The results~\cite[Theorem 3.2.1]{Faraut:08},~\cite[Theorem
5.20]{Bullo:05} show that $\Lie(\g)$ is a subalgebra of
$\M{n,\Real}$. The Lie algebras encountered in this paper are given by
\begin{equation*}
\begin{aligned}
&\Lie(\GL{n,\Real}) \coloneqq \M{n,\Real} \\
&\Lie(\SL{n, \Real}) \coloneqq \{ A \in \M{n,\Real} : \tr(A) = 0 \} \\ 
&\Lie(\SO{n, \Real}) \coloneqq \{ A \in \M{n,\Real} : A + A^{\top} = 0_n \} .
\end{aligned}
\end{equation*}
\begin{definition}
  Let $\g$ be a linear Lie group. A one-parameter subgroup of $\g$ is a
  continuous group homomorphism $\gamma: \Real^{+} \rightarrow \g$.
\end{definition}
\begin{lemma}[{\cite[Theorem 3.1.1]{Faraut:08}}]
\label{lemma:oneparam_subgroup}
If $\gamma: \Real^{+} \rightarrow \GL{n,\Real}$ is a one-parameter
subgroup of $\GL{n,\Real}$, then $\gamma$ is real-analytic and
$\gamma(t) = \exp(t A)$, with $A = \gamma^{\prime}(0)$.
\end{lemma}
\begin{lemma}
\label{lemma:inverse_deriv}
Let $X: \Real \rightarrow \GL{n,\Real}$ be a smooth parameterized
curve. Then
\begin{equation}
\begin{aligned}
\frac{\D }{\D\tau} X^{-1}&= - X^{-1}
\frac{\D X}{\D\tau}X^{-1}.
\end{aligned}
\label{eq:dinv}
\end{equation}
\end{lemma}
\begin{proof}
  For any smooth parameterized curve $X: \Real \rightarrow \GL{n,\Real}$, 
  and any $\tau \in \Real$, $X(\tau) X^{-1}(\tau) = I_n$ where
  $\tau$ is the curve parameter. Differentiating both sides of this
  identity with respect to $\tau$, while using the product rule, we
  get
\begin{equation*}
\begin{aligned}
  0 = \frac{\D}{\D\tau} I_n & = \frac{\D}{\D\tau} \left( X X^{-1} \right) \\
  &= \frac{\D X}{\D\tau}X^{-1} + X \frac{\D X^{-1} }{\D\tau},
\end{aligned}
\end{equation*}
from which~\eqref{eq:dinv} immediately follows.
\end{proof}

\begin{definition}
Given a smooth parameterized curve $X: \Real \rightarrow \GL{n,\Real}$, 
the body-velocity of $X$ is the curve $v_X: \Real \rightarrow \M{n,\Real}$, 
defined by $v_X(t) \coloneqq X^{-1}(t) \dot{X}(t)$.
\end{definition}

\subsection{Matrix logarithms}
\begin{definition}
Given a matrix $X \in \Real^{n \times n}$, any $n \times n$ matrix $A$
such that
\[
X = \exp{(A)}
\]
is a logarithm of $X$.
\label{def:mlog}
\end{definition}
Every $X \in \GL{n, \Real}$ has a logarithm~\cite{Wou65}, but the
logarithm isn't necessarily real. Since this paper deals with real Lie
groups, we are interested in conditions under which a matrix has a
real logarithm. Furthermore, to ensure uniqueness, we henceforth only
consider the principle logarithm of a matrix. The following result is
taken from~\cite[Theorem 1.31]{Higham:08},~\cite{Cul66}.
\begin{theorem}[principle logarithm]
  If $X \in \mathbb{\Real}^{n \times n}$ has no eigenvalues on
  $\Real_{\leq 0}$ then there is a unique, real, logarithm $A$ of $X$,
  all of whose eigenvalues lie in the strip
\[
\set{z \in \mathbb{C} : -\pi < \IM{(z)} < \pi}.
\]
\label{thm:plog}
\end{theorem}
We refer to the logarithm $A$ in Theorem~\ref{thm:plog} as the
principle logarithm of $X$ written $A = \log{(X)}$. The principle
logarithm can equivalently be defined as a series.

\begin{definition}
  Given a matrix $X \in \Real^{n \times n}$ with no eigenvalues on
  $\Real_{\leq 0}$, the principle matrix logarithm is
\begin{equation}
\begin{aligned}
\log(X) \coloneqq\sum_{k=1}^{\infty} \frac{(-1)^{k+1}}{k} (X-I_n)^k .
\label{eq:logarithm_def}
\end{aligned}
\end{equation}
\end{definition}
The series definition of the principle matrix logarithm
~\eqref{eq:logarithm_def} converges for every matrix $X \in B(I_n,1)
\subset \GL{n,\Real}$ because $\sum_k \frac{1}{k}\|X - I_n\|^k$
converges for $\|X - I_n\| < 1$. An alternative series representation
of the principle matrix logarithm is given by Gregory's series
(1668)~\cite[Section 11.3]{Higham:08},
\begin{equation*}
\log(X) = -2 \sum_{k=0}^{\infty} \frac{1}{2 k + 1} 
\left( \left(I_n - X\right) \left(I_n + X\right)^{-1}   \right)^{2 k + 1}.
\end{equation*}
This series converges if all the eigenvalues of $X$ have strictly
positive real parts, though the rate of convergence may be
slow. Regardless of which series is used, the structure of the
proposed observers remains the same and only the proven region of
convergence changes.

The next elementary result summarizes some important properties of
$\exp$ and $\log$ used in this paper.
\begin{lemma}
\label{lemma:explog_properties}
The exponential and the principle logarithm maps have the following
properties
\begin{itemize}
\item[(a)] For all $X \in \Real^{n \times n}$ with $\sigma(X) \cap
  \Real_{\leq 0} = \varnothing$
\[
\exp(\log(X)) = X,
\]
\item[(b)] For all $A \in B(0_n, \log(2))$,
\[
\log(\exp(A)) = A,
\]
\item[(c)] For all $A \in \Real^{n \times n}$ and all $X \in
  \GL{n,\Real}$,
\[
\exp(X A X^{-1}) = X \exp(A) X^{-1},
\]
\item[(d)] For all $X \in \Real^{n \times n}$ with $\sigma(X) \cap
  \Real_{\leq 0} = \varnothing$, and all $A \in \GL{n, \Real}$ such
  that $\sigma(A X A^{-1}) \cap \Real_{\leq 0} = \varnothing$ and
  $\sigma(A X A^{-1}) \subset \set{z \in \mathbb{C} : -\pi < \IM{(z)}
    < \pi}$
\[
\log(A X A^{-1}) = A \log(X) A^{-1}.
\]
\end{itemize}
\end{lemma}
\begin{proof}
  Property (a) is a direct consequence of Definition~\ref{def:mlog}
  and Theorem~\ref{thm:plog}. To see this note that since $\sigma(X)
  \cap \Real_{\leq 0} = \varnothing$, there is a unique $A \in
  \Real^{n \times n}$ with $\sigma{(A)} \subset \set{z \in \mathbb{C}
    : -\pi < \IM{(z)} < \pi}$ such that $X = \exp(A)$, namely, $A =
  \log{(X)}$. The proof of property (b) is standard and can be found
  in~\cite[Theorem 2.2.1]{Faraut:08} or~\cite{Ros02}. To see that (c)
  holds, note that $(X A X^{-1})^k = X A^k X^{-1}$.  Substituting this
  identity into the series definition~\eqref{eq:exponential_def} of
  the exponential map gives the desired result.

  To show that (d) holds, first by property (a) and the hypothesis on
  $AXA^{-1}$
\[
\exp{(\log(A X A^{-1}))} = AXA^{-1}.
\]
Second, by properties (a) and (c) and the hypothesis on
  $X$
\[
\exp{(A\log(X)A^{-1})} = A\exp{(\log(X))}A^{-1} = AXA^{-1}.
\]
By Theorem~\ref{thm:plog} and the assumption on $AXA^{-1}$, the
equation
\[
\exp{(C)} = AXA^{-1}
\]
has a unique real solution $C$ with $\sigma(C) \subset \set{z \in
  \mathbb{C} : -\pi < \IM{(z)} < \pi}$. Therefore, we conclude that
$\log(A X A^{-1}) = A \log(X) A^{-1}$ as required.
\end{proof}
The next result is a direct consequence of~\cite[Theorem 1]{Wou65}. It
plays an important role in our analysis.
\begin{theorem}[\cite{Wou65}]
  Let $A \in \Real^{n \times n}$ be such that $\sigma(A) \cap
  \Real_{\leq 0} = \varnothing$. Then $\log{(A)}$ commutes with any
  matrix that commutes with $A$.
\label{thm:logcommute}
\end{theorem}
%

If $\g \subseteq \GL{n,\Real}$ is any linear Lie group, then a
consequence of Lemma~\ref{lemma:explog_properties}, and of the
definition of $\Lie(\g)$, the map $\exp : \Lie(\g) \to \g$ is a local
diffeomorphism of zero in $\Lie(\g)$ onto a \nbhd of $I_n$ in
$\g$. Thus, restricted to a sufficiently small \nbhd $U$ in $\g$ of
$I_n$, the matrix logarithm $\log : U \subset \g \to \Lie(\g)$ is the
inverse of $\exp : \Lie(\g) \to \g$. Furthermore $\g \cap B(I_n ,1)
\subseteq U$.

\subsection{The tangent space of a linear Lie group}
By definition a Lie group is a differentiable manifold. Therefore we
can define the tangent space at a point in the group as an equivalence
class of curves.
\begin{definition}
  Let $\g$ be a linear Lie group and let $X \in \g$. A curve at $X$ is
  a $C^1$ map $c : I \to \g$, $t \mapsto c(t)$, from an open interval
  $I \subseteq \Real$ with $0 \in I$ and $c(0) = X$. Let $c_1$ and
  $c_2$ be two curves at $X$ and $(W, \psi)$, $W \subseteq \g$, a
  chart on $\g$ with $X \in W$. Then, $c_1$ and $c_2$ are tangent at
  $X$ with respect to $\psi$ if
\[
\left.\frac{\D \left(\psi \circ c_1\right)}{\D t}\right|_{t = 0} =
\left.\frac{\D \left(\psi \circ c_2\right)}{\D t}\right|_{t = 0}.
\]
\end{definition}
In other words, two curves are tangent at a point $X \in \g$ if their
tangent vectors in local coordinates are equal. Tangency at $X$ is a
coordinate-independent notion and defines an equivalence relation
among curves at $X$. Let $\left[c\right]_X$ denote one such
equivalence class. 
\begin{definition}
  Let $\g$ be a linear Lie group and let $X \in \g$. The tangent space
  to $\g$ at $X$, $T_X\g$ is the set of equivalence classes at $X$:
\[
T_X \g \coloneqq \{ \left[c\right]_X : \text{$c$ is a curve at $X$}\}.
\]
Each equivalence class $\left[c\right]_X$ is a tangent vector at $X$.
\end{definition}
In the case of linear Lie groups the set of equivalence classes
$\left[c\right]_X$ can be characterized in a particularly simple
manner.
\begin{proposition}
\label{prop:lie_tangent_space}
Let $\g$ be a linear Lie group and let $X \in \g$. Then
\begin{equation*}
\begin{aligned}
T_X \g &= X \Lie(\g) \coloneqq \{ X A : A \in \Lie(\g) \} \\
&= \Lie(\g) X \coloneqq \{ A X : A \in \Lie(\g) \}  . 
\end{aligned}
\end{equation*}
\end{proposition}
\begin{proof}
By Lemma~\ref{lemma:explog_properties}, part (c), we have that 
\[
\left( \forall X \in \g \right) \; 
\left( \forall A \in \Lie(\g) \right) \; 
\left( \forall t \in \Real \right) \; \hspace{3mm} 
X \exp(tA) X^{-1} = \exp(t X A X^{-1}) .
\]
Thus, by Definition~\ref{def:liealg}, $X \Lie(\g) X^{-1} \subseteq
\Lie(\g)$. Using an identical argument, replacing $X$ with $X^{-1}$,
one verifies that $X^{-1} \Lie(\g) X \subseteq \Lie(\g)$, which
implies $\Lie(\g) \subseteq X \Lie(\g) X^{-1}$. Therefore, we have
shown that $X \Lie(\g) X^{-1} = \Lie(\g)$, which proves that $X
\Lie(\g) = \Lie(\g) X$.

Now, for any $A \in \Lie(\g)$, the curve $c(t) = X \exp(t A)$ is
smooth with $\gamma(0) = X$, hence $\dot{\gamma}(0) \in T_X \g$.
Computing the derivative of the curve
\[
\dot{\gamma}(0) = \frac{\D}{\D t} X \exp(t A) \big|_{t=0} = X A ,  
\]
which shows that $X A \in T_X \g$. Since our choice of 
$A \in \Lie(\g)$ was arbitrary, we have that 
$X \Lie(\g) \subseteq T_X \g$. 

Conversely, for any $B \in T_X \g$, by definition there exists a smooth
curve $c : I \to \g$, with $c(0) = X$ and $\dot{c}(0) = B$.  For
small $t \in I$, $|t|$ sufficiently small, define the curve
\[
\begin{aligned}
\beta : & \; I \to \Lie(\g)\\
& \; t \mapsto \log(X^{-1} c(t)) , 
\end{aligned}
\]
which satisfies $\beta(0) = 0_n$.  Furthermore, since $\Lie(\g)$ is a
vector space, we have that $\dot{\beta}(0) \in T_{\beta(0)}\Lie(\g)
\simeq \Lie(\g)$.  For small $\left| t \right|$, the curve $c$
is given by
\[
c(t) = X \exp(\beta(t)).
\]
Computing $\dot{c}(0)$, and using the fact that $\beta(0) = 0$, we get
\begin{equation*}
\begin{aligned}
\dot{\gamma}_B(0)
&= X \frac{\D}{\D t} \left[ \sum_{k=0}^{\infty} \frac{\beta(t)^k}{k!} \right]
\Bigg|_{t=0} \\
&= X \frac{\D}{\D t} \left[ I_n + \beta(t) + \frac{\beta(t)^2}{2} + \cdots \right] 
\Bigg|_{t=0} \\
&= X \left[\dot{\beta}(t) + \frac{\dot{\beta}(t) \beta(t) + \beta(t) \dot{\beta}(t)}{2} + 
\cdots \right] \Bigg|_{t=0} \\
&= X  \dot{\beta}(0) .
\end{aligned}
\end{equation*}
Therefore $B = \dot{c}(0) = X \dot{\beta}(0)$. However, since
$\dot{\beta}(0) \in \Lie(\g)$, we conclude that $B \in X \Lie(\g)$.
Since our choice of $B \in T_X \g$ was arbitrary, we have that $T_X \g
\subseteq X \Lie(\g)$.
\end{proof}

\section{Problem Statements}
Partially motivated by the discussion from
Section~\ref{sec:motivation} we introduce the two problems studied in
this paper. The first problem deals with kinematic systems on linear
Lie groups while the second relates to dynamic systems on linear Lie
groups.

\subsection{Full state observers}
\label{sec:fullproblem}
Let $\g \subseteq \GL{n,\Real}$ be a linear Lie group.  Consider the
following system on $\g$
\begin{equation}
\begin{aligned}
\dot{X} &= X u \\
Y &= X ,
\label{eq:leftinv_system_def}
\end{aligned}
\end{equation}
where $u: \Real \rightarrow \Lie(\g)$ is the control input to the
system, and $Y \in \g$ is the measured output of the
system. System~\eqref{eq:leftinv_system_def} is left-invariant.  This
means that, for any fixed matrix $A \in \g$, if we redefine the state
as $Z \coloneqq A X$, then the new state $Z$ satisfies the same
differential equation as $X$, i.e., $\dot{Z} = Z u$.

In this paper we assume that the control signal $u$ is admissible for
the system~\eqref{eq:leftinv_system_def}.  This means that for any
initial condition $X(0) \in \g$, the corresponding solution
of~\eqref{eq:leftinv_system_def} with the admissible input $u$ is
unique, continuously differentiable and exists for all time.

\begin{assumption}
  The input $u$ to system~\eqref{eq:leftinv_system_def} is such that
  for any initial condition $X(0) \in \g$ the corresponding solution
  $X(t)$ is bounded.
\label{ass:bounded}
\end{assumption}
Assumption~\ref{ass:bounded} is automatically satisfied if the group
$\g$ is compact, for example $\g = \SO{3, \Real}$.

\begin{problem}
  Given a left-invariant system~\eqref{eq:leftinv_system_def} on a
  linear Lie group $\g \subseteq \GL{n,\Real}$ with input $u \in
  \Lie(\g)$ such that Assumption~\ref{ass:bounded} holds,
  design a state estimator with estimate $\hat{X} \in \g$, access to
  $Y \in \g$ and $u \in \Lie(\g)$, such that, for $\hat{X}(0)$
  sufficiently close to $X(0)$, $\hat{X}(t) \longrightarrow X(t)$
  exponentially, as $t \rightarrow \infty$.
\label{prob:kinematic}
\end{problem}
\begin{figure}[ht]
\centering
\includegraphics[width=0.5\textwidth]{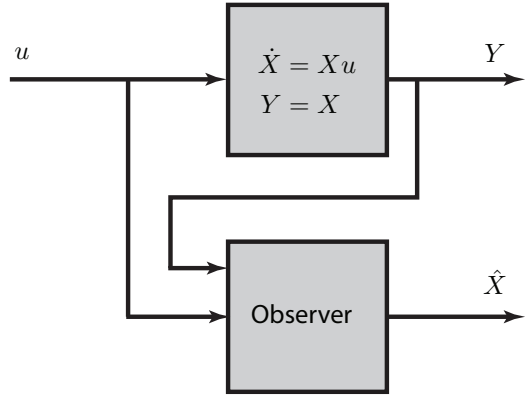}
\caption{Block diagram of the full state observer setup.}
\label{fig:block_diagram}
\hspace{0.5cm}
\end{figure}

We emphasize that the results of this paper can be extended to
right-invariant systems on Lie groups, i.e., systems of the form
\begin{equation*}
\begin{aligned}
\dot{X} &= u X .
\label{eq:rightinv_system_def}
\end{aligned}
\end{equation*}
However, we restrict the discussion to left-invariant systems to avoid
repetition and for clarity.

\subsection{Partial state observers}
\label{sec:partialproblem}
If $\g$ is any linear Lie group then, as we have already seen, its Lie
algebra $\Lie(\g)$, is a vector space and a subalgebra of
$\M{n,\Real}$. Therefore the tangent space at any point $X \in G$ is
isomorphic to the Lie algebra itself.  This implies that if $c: \Real
\rightarrow \Lie(\g)$ is any smooth curve, then its derivative,
$\dot{x}$, is also a curve in $\Lie(\g)$, i.e., $\dot{c}: \Real
\rightarrow \Lie(\g)$.  Consider the following system
\begin{equation}
\begin{aligned}
\dot{X} &= X x_2 \\
\dot{x}_2 &= x_3 \\
&\;\;\vdots \\
\dot{x}_d & = u \\
Y &= X ,
\label{eq:LPSO_system_def}
\end{aligned}
\end{equation}
where $X \in \g$ evolves on a linear Lie group and $x_i \in \Lie(\g)$,
$i \in \left\{2, \ldots, d\right\}$. The input
to~\eqref{eq:LPSO_system_def} is $u: \Real \rightarrow \Lie(\g)$ which
we assume to be a smooth, uniformly bounded and locally Lipschitz
signal of time.

The following assumption, which is almost identical to
Assumption~\ref{ass:bounded} in Section~\ref{sec:fullproblem},
is assumed throughout this paper.
\begin{assumption}
  The input $u$ to system~\eqref{eq:LPSO_system_def} is such that for
  any initial condition $X(0) \in \g$, $x_2(0), \ldots, x_d(0) \in
  \Lie(\g)$, the corresponding solution $(X(t), x_2(t), \ldots,
  x_d(t))$ of~\eqref{eq:LPSO_system_def} is such that $X(t)$ is
  bounded.
\label{ass:dyn_bounded}
\end{assumption}

\begin{problem}
  Given the system~\eqref{eq:leftinv_system_def}, design a state
  estimator with estimate $\hat{X} \in \g$, $\hat{x}_i \in
  \Lie{(\g)}$, $i \in \left\{2,\ldots,d\right\}$, access to the output
  $Y \in \g$ and the input $u \in \Lie(\g)$, such that, under
  Assumption~\ref{ass:dyn_bounded}, if $\|\hat{X}(0) - X(0)\|$,
  $\|\hat{x}_2(0) - x_2(0)\|$, $\ldots$, $\|\hat{x}_d(0) - x_d(0)\|$
  are sufficiently small, then $\|\hat{X}(t) - X(t)\| \rightarrow 0$,
  $\|\hat{x}_2(t) - x_2(t)\| \rightarrow 0$, $\ldots$, $\|\hat{x}_d(t)
  - x_d(t) \| \rightarrow 0$ exponentially, as $t \rightarrow \infty$.
\label{prob:dynamic}
\end{problem}

\section{Proposed observers}
\label{sec:proposed}
In this section we propose various observers that solve
Problems~\ref{prob:kinematic} and~\ref{prob:dynamic}. The analysis of
the observers is presented in Section~\ref{sec:errordyn} where, using
the results of Section~\ref{sec:ODE}, we provide conditions under
which the observers solve Problems~\ref{prob:kinematic}
and~\ref{prob:dynamic}.

\subsection{Local full state observers}
\label{sec:LFSO_observer}
For the system \eqref{eq:leftinv_system_def}, we propose two different
observers, which we call local Lie group Full State Observers (LFSOs).
The first is the passive LFSO, given by
\begin{equation}
\begin{aligned}
\dot{\hat{X}} &= \hat{X} u - a_0 \hat{X} \log(Y^{-1} \hat{X}) .
\label{eq:obs_passive1}
\end{aligned}
\end{equation}
The second is the direct LFSO, given by
\begin{equation}
\begin{aligned}
\dot{\hat{X}} &= Y u Y^{-1} \hat{X} - a_0 \hat{X} \log(Y^{-1}
\hat{X}) .
\label{eq:obs_direct1}
\end{aligned}
\end{equation}
In the above two observers, the constant $a_0 > 0$ is a design
parameter that, as we will show, can be used to change the rate of
observer convergence. Following the terminology
of~\cite{LagTruMah:10}, we call the term $\alpha(\hat{X}, Y) \coloneqq
- a_0 \hat{X} \log(Y^{-1} \hat{X})$, appearing in
\eqref{eq:obs_passive1} and \eqref{eq:obs_direct1}, the innovation
term of the observer. It can be verified that the term $\alpha$
satisfies the definition, given in~\cite[Definition 15]{LagTruMah:10},
of an innovation term.

An intuitive and informal explanation for taking this particular form
of $\alpha(\hat{X}, Y) = - a_0 \hat{X} \log(Y^{-1} \hat{X})$ is that
the matrix $Y^{-1} \hat{X}$ represents a ``measurement error" on the
Lie group $\g$. The Lie group $\g$ is not a vector space and therefore
$Y^{-1} \hat{X}$ is not a vector and should not be added or subtracted
with other matrices. To make the matrix $Y^{-1} \hat{X}$ more
vector-like, we take its logarithm, which maps it to the vector space
$\Lie(\g)$, while preserving all its information (since $\log: G
\rightarrow \Lie(\g)$ is a local diffeomorphism).  We then
push-forward this vector, $\log(Y^{-1} \hat{X})$, from $T_{I_n}G$ to
$T_{\hat{X}}G$, the tangent space at $\hat{X}$, by applying the
push-forward map of left translation.  By
Proposition~\ref{prop:lie_tangent_space}, left translation from
$T_{I_n} G$ to $T_{\hat{X}} G$ is the same as left multiplication by
$\hat{X}$.  The result is a vector in $T_{\hat{X}} G$, which
represents the estimation error. This vector-like estimation error is
then multiplied by the gain $a_0 > 0$ to adjust the rate of observer
convergence.

We call the term $\hat{X} u$ appearing in~\eqref{eq:obs_passive1} and
the term $Y u Y^{-1} \hat{X}$ appearing in~\eqref{eq:obs_direct1}, the
synchronization terms of the observer. The choice of synchronization
term distinguishes the passive observer from the direct observer.

It is computationally costly and inefficient to compute the matrix
logarithm map using the series definitions. Various studies have
looked at the problem of approximating this computation. In particular
the
work~\cite{cheng2001approximating},~\cite{kenney1998schur},~\cite{higham2001evaluating},~\cite{cardoso2006pade}
may be useful for implementing the observers proposed in this
paper. While we do not pursue the notion of using approximations to
the matrix logarithm to implement the observers, we do observe, in
Section~\ref{sec:examples}, that in the special case $\g =
\SO{3,\Real}$ the logarithm can be computed efficiently.

\subsection{Local partial state observers}
\label{sec:LPSO_observer}

For system \eqref{eq:LPSO_system_def}, we propose two different
observers which we call local Lie group Partial State Observers
(LPSOs). The first is the direct LPSO, given by
\begin{equation}
\begin{aligned}
\dot{\hat{X}} &= Y \hat{x}_2 Y^{-1} \hat{X} - a_{d-1} \hat{X} \log(Y^{-1} \hat{X}) \\
\dot{\hat{x}}_2 &= \hat{x}_3 - a_{d-2} \log(Y^{-1} \hat{X}) \\
\vdots \\
\dot{\hat{x}}_{d-1} &= \hat{x}_d - a_{1} \log(Y^{-1} \hat{X}) \\
\dot{\hat{x}}_d &= u - a_{0} \log(Y^{-1} \hat{X})   
\label{eq:LPSO_obs_direct1}
\end{aligned}
\end{equation}
and the second is the passive LPSO, given by
\begin{equation}
\begin{aligned}
\dot{\hat{X}} &= \hat{X} \hat{x}_2 - a_{d-1} \hat{X} \log(Y^{-1} \hat{X}) \\
\dot{\hat{x}}_2 &= \hat{x}_3 - a_{d-2} \log(Y^{-1} \hat{X}) \\
\vdots \\
\dot{\hat{x}}_{d-1} &= \hat{x}_d - a_{1} \log(Y^{-1} \hat{X}) \\
\dot{\hat{x}}_d &= u - a_{0} \log(Y^{-1} \hat{X}) .
\label{eq:LPSO_obs_passive1} 
\end{aligned}
\end{equation}
In the above two observers, the constants $a_0, \hdots, a_{d-1} \in
\Real$ are design parameters, chosen such that the polynomial $p(s) =
s^d + a_{d-1} s^{d-1} + \cdots + a_1 s + a_0$ is Hurwitz.  These
design parameters can be used to modify the rate of convergence of the
estimation error.

We will show that the direct LPSO~\eqref{eq:LPSO_obs_direct1} is
locally exponentially stable for the
system~\eqref{eq:LPSO_system_def}. Unfortunately we are not able to
show exponential stability of the passive
LPSO~\eqref{eq:LPSO_obs_passive1}. However, we will see in simulation
that the passive LPSO works well in the presence of measurement noise.

\section{Estimation Error Functions}

Following~\cite{LagTruMah:10}, we introduce two canonical choices of
estimation error functions for left-invariant systems on Lie groups.
\begin{definition}
\label{def:LRerror}
  Given system~\eqref{eq:leftinv_system_def} with $X \in \g$, and an
  observer with state estimate $\hat{X} \in \g$, the canonical
  left-invariant error, $E_l: \g \times \g \rightarrow \g$, is
\begin{equation*}
E_{l}(X, \hat{X}) \coloneqq X^{-1} \hat{X} 
\label{eq:leftinv_error_def}
\end{equation*}
and the canonical right-invariant error, 
$E_r: \g \times \g \rightarrow \g$, is
\begin{equation*}
E_{r}(X, \hat{X}) \coloneqq \hat{X} X^{-1} .
\label{eq:rightinv_error_def}
\end{equation*}
\end{definition}
The error $E_l$ is called left-invariant because, for any $A \in \g$,
$E_l(A X, A \hat{X}) = X^{-1} A^{-1} A \hat{X} = X^{-1} \hat{X} =
E_l(X, \hat{X})$.  Similarly, $E_r$ is called right-invariant since,
for any $A \in \g$, $E_r(X A, \hat{X} A) = E_r(X, \hat{X})$. From the
definitions of $E_l$ and $E_r$, it is apparent that
\begin{equation}
\label{eq:ElEr}
E_r = \Ad_X(E_l) = X E_l X^{-1}
\end{equation}
where, for any fixed $X \in \g$, $\Ad_X: \g \rightarrow \g$ is a global
diffeomorphism.


In Problems~\ref{prob:kinematic} and~\ref{prob:dynamic} we seek to
design observers so that $\| \hat{X} - X \| \rightarrow 0$
exponentially. To characterize this property we rely on the following
result.

\begin{proposition}
\label{prop:norm_convergence}
Suppose that $X : \Real \to \g$ is bounded. If either $E_r \rightarrow
I_n$ exponentially, or $E_l \rightarrow I_n$ exponentially, as $t
\rightarrow \infty$, then $\hat{X} \rightarrow X$ exponentially, as $t
\rightarrow \infty$.
\end{proposition}

\begin{proof}
  For any $X$, $\hat{X} \in \g$, using Definition~\ref{def:LRerror},
  the following identities hold
\begin{equation*}
\begin{aligned}
\hat{X} - X &= X (X^{-1} \hat{X} - I_n) = X (E_l - I_n) \\
\hat{X} - X &= (\hat{X} X^{-1} - I_n) X = (E_r - I_n) X .
\label{eq:lemma_norm_convergence_identities}
\end{aligned}
\end{equation*}
Taking the norms of these identities, we obtain
\begin{equation}
\begin{aligned}
\| \hat{X} - X \| &\leq \| X \| \|E_l - I_n\| \\
\| \hat{X} - X \| &\leq \| X \| \| E_r - I_n \|  .
\label{eq:lemma_norm_convergence_inequalities}
\end{aligned}
\end{equation}
Additionally, for any $ X, \hat{X} \in G$,
\begin{equation*}
\begin{aligned}
(E_l - I_n) &= X^{-1} (\hat{X} - X) \\
(E_r - I_n) &= (\hat{X} - X) X^{-1} 
\label{eq:lemma_norm_convergence_identities2}
\end{aligned}
\end{equation*}
so that
\begin{equation}
\begin{aligned}
\| E_l - I_n \| &\leq \| X^{-1} \| \| \hat{X} - X \| \\
\| E_r - I_n \| &\leq \| X^{-1} \| \| \hat{X} - X \|.
\label{eq:lemma_norm_convergence_inequalities2}
\end{aligned}
\end{equation}
By hypothesis, $\| X(t) \|$ is uniformly bounded, i.e., $\left(\exists
  K_1 > 0\right)$$\left(\forall t \geq 0\right)$$\|X(t)\| \leq
K_1$. This implies that $X(t)$ evolves on the compact subset
$\mathcal{G} = \left\{X \in \GL{n, \Real} : \|X\| \leq
  K_1\right\}$. Since the matrix inverse map is continuous, the image
of $\mathcal{G}$ under the matrix inverse map is also a compact subset
of $\GL{n,\Real}$. Therefore, $\| X^{-1}(t) \|$ is also uniformly
bounded, i.e., $\left(\exists
  K_2 > 0\right)$$\left(\forall t \geq 0\right)$$\|X^{-1}(t)\| \leq
K_2$.

Now suppose that $\| E_r(t) - I_n \| \rightarrow 0$ exponentially, as
$t \rightarrow \infty$, then by the definition of exponential
stability, we have
\begin{equation*}
\begin{aligned}
  &\left( \exists \delta, m , \alpha > 0 \right) \; \left( \forall
    E_r(0) \in B(I_n, \delta) \right)\left( \forall t \geq 0 \right)
  \; \| E_r(t) - I_n \| < m e^{-\alpha t} \| E_r(0) - I_n\|.
\end{aligned}
\end{equation*}
By the inequalities~\eqref{eq:lemma_norm_convergence_inequalities}, 
and uniform boundedness of $\| X \|$, we have that 
\begin{equation*}
\begin{aligned}
\| E_r - I_n \| < m  \Rightarrow \| \hat{X} - X \| < K_1 m. 
\end{aligned}
\end{equation*}
By the inequalities~\eqref{eq:lemma_norm_convergence_inequalities2}, 
and uniform boundedness of $\| X^{-1} \|$, we have that 
\begin{equation*}
\begin{aligned}
\| \hat{X} - X \| < \frac{\delta}{K_2} \Rightarrow \| E_r - I_n \| < \delta . 
\end{aligned}
\end{equation*}
Combining the above results, we have exponential convergence of 
$\| \hat{X} - X \| \rightarrow 0$, 
\begin{equation*}
\begin{aligned}
  &\left( \exists \delta, m , \alpha > 0 \right) \left(\|
    \hat{X}(0) - X(0) \| < \delta / K_2 \right)\\ &\qquad\qquad \left(
    \forall t \geq 0 \right) \; \| \hat{X}(t) - X(t) \| < K_1 K_2 m
  e^{-\alpha t} \| \hat{X}(0) - X(0)\|.
\end{aligned}
\end{equation*}
The proof for $E_l$ is identical.
\end{proof}

In addition to the error functions $E_l$ and $E_r$, we introduce two
other, closely related, error functions.

\begin{definition}
  For any $E_l \in B(I_n,1)$, the log left-invariant error, $e_l: \g
  \times \g \rightarrow \Lie(\g)$, is
\begin{equation*}
e_l(X, \hat{X}) \coloneqq \log(E_{l}(X, \hat{X})) = \log(X^{-1}
\hat{X}).
\end{equation*}
For any $E_r \in B(I_n,1)$, the log right-invariant error, $e_r: \g
\times \g \rightarrow \Lie(\g)$, is
\begin{equation*}
e_r(X, \hat{X}) \coloneqq \log(E_{r}(X, \hat{X})) = \log(\hat{X}
X^{-1}) .
\end{equation*}
\end{definition}
Since $e_l$ is solely a function of $E_l$, and since $E_l$ is
left-invariant, it follows that $e_l$ is also left-invariant, i.e.
$\forall A \in \g : \; e_l(A X, A \hat{X}) = e_l(X, \hat{X})$. For the
same reasoning, it follows that $e_r$ is right-invariant, i.e.
$\forall A \in \g : \; e_r(X A, \hat{X} A) = e_r(X, \hat{X})$.

The variables $e_l$ and $e_r$ are useful because they are vectors in
$\Lie(\g)$ and they allow us to convert a differential equation on a
Lie group into a differential equation on a vector space. The
disadvantage of $e_l$ and $e_r$ is that they are only defined for
$E_l$, $E_r \in B(I_n,1)$.

\begin{lemma}
\label{lemma:log_error_relationship}
If $E_l, E_r \in B(I_n,1)$, then
\begin{equation*}
\begin{aligned}
e_r &= X e_l X^{-1} .
\label{eq:lemma_logerror_related}
\end{aligned}
\end{equation*}
\end{lemma}
\begin{proof}
With the help of Lemma
\ref{lemma:explog_properties} (d), we obtain
\begin{equation*}
\begin{aligned}
e_r &= \log(E_r) \\
&= \log(X E_l X^{-1}) \\
&= X \log(E_l) X^{-1} \\
&= X e_l X^{-1} .
\end{aligned}
\end{equation*}
\end{proof}
Finally, in the context of partial state observers, since $x_i$ and
$\hat{x}_i$, for $i = 2, \ldots, d$ are vectors in $\Lie(\g)$, to
quantify the error between $x_i$ and $\hat{x}_i$, we can use
subtraction of vectors
\begin{equation}
\begin{aligned}
e_i \coloneqq x_i - \hat{x}_i , \hspace{3mm} i = 2, \hdots, d
\label{eq:LPSO_error_vec_def} .
\end{aligned}
\end{equation}
Since $x_i$ and $\hat{x}_i$ are elements of the vector space
$\Lie(\g)$, $e_i$ is also an element of $\Lie(\g)$.

\section{Differential Equations on Matrices}
\label{sec:ODE}

In this section we study the properties of a pair of differential
equations on linear Lie groups. These differential equations arrise in
the analysis of the error dynamics associated with the observers
proposed in Section~\ref{sec:proposed}.

\subsection{A Differential Equation on $\GL{n,\Real}$}
\label{sec:LFSO_de_matrix}

Consider the differential equation evolving on $\GL{n,\Real}$ given by
\begin{equation}
\begin{aligned}
\dot{E} &= -a_0 E \log(E) ,
\label{eq:basic_eqn}
\end{aligned}
\end{equation}
where $E \in \GL{n,\Real}$ and $a_0 \in \Real$ is a positive
constant. The equation~\eqref{eq:basic_eqn} arises in the analysis of
the error dynamics associated with the
observers~\eqref{eq:obs_passive1},~\eqref{eq:obs_direct1}. Note that,
by Theorem~\ref{thm:logcommute}, the above equation can be written
$\dot{E} = - a_0 \log(E) E$.


While we have defined~\eqref{eq:basic_eqn} to evolve on the set of all
invertible matrices, $\GL{n,\Real}$.  If $\g \subseteq \GL{n,\Real}$
is any linear Lie group, then the vector field~\eqref{eq:basic_eqn} is
tangent to the submanifold $\g$.  Therefore, the submanifold $\g$ is
positively invariant for~\eqref{eq:basic_eqn}.  To see that the vector
field~\eqref{eq:basic_eqn} is tangent to any linear Lie group $\g$,
suppose that $E(t_0) \in \g$ at some time $t_0 \in \Real$.  Then
$\log(E(t_0)) \in \Lie(\g)$ and left-translation of this vector takes
it to the tangent space to $\g$ at $E(t_0)$, i.e., $E(t_0)
\log(E(t_0)) \in T_{E(t_0)}\g$, by
Proposition~\ref{prop:lie_tangent_space}.  Thus, the vector
field~\eqref{eq:basic_eqn} is such that $\dot{E}(t_0) \in
T_{E(t_0)}\g$.


The crucial property of the differential equation~\eqref{eq:basic_eqn}
is that the matrices $\dot{E}$ and $E$ commute,
i.e., $E \dot{E} = \dot{E} E$. This property is a consequence of
matrices $E$ and $\log(E)$ commuting. Commutativity of
$\dot{E}$ and $E$, combined with the product rule, gives us the
following result.

\begin{lemma}
\label{lemma:deriv_of_product}
Let $E: \Real \rightarrow \GL{n,\Real}$ be a curve in $\GL{n,\Real}$,
such that $E$ and $\dot{E}$ commutes. Then for all positive integers
$k$
%
%
%
\begin{equation*}
\frac{\D}{\D t} \left[ (E-I_n)^k \right] = k \dot{E} (E-I_n)^{k-1} = k (E-I_n)^{k-1} \dot{E} .
\label{eq:deriv_of_product}
\end{equation*}
\end{lemma}
\begin{proof}
By straight-forward computation, using the product rule and
commutativity of $\dot{E}$ and $E$, we get
\begin{equation*}
\begin{aligned}
\frac{\D}{\D t} \left[ (E-I_n)^{k} \right] &= \frac{\D}{\D t} \left[
\overbrace{ (E-I_n)(E-I_n)\cdots(E-I_n) }^{\text{k times }} \right] \\
&= \overbrace{ \dot{E} (E-I_n)^{k-1} + 
(E-I_n) \dot{E} (E-I_n)^{k-2} + \cdots + (E-I_n)^{k-1} \dot{E} }^{\text{k
times}} \\
&= k \dot{E} (E-I_n)^{k-1} \\
&= k (E-I_n)^{k-1} \dot{E} .
\label{eq:deriv_of_product_proof_1}
\end{aligned}
\end{equation*}
\end{proof}
Lemma~\ref{lemma:deriv_of_product} is the key reason
why~\eqref{eq:basic_eqn} is differentially equivalent to a linear
differential equation.  The change of coordinates that realizes this
equivalence is the matrix logarithm map defined on $B(I_n,1) \subset
\GL{n,\Real}$. By Lemma~\ref{lemma:explog_properties}, the matrix
logarithm map $\log: B(I_n,1) \rightarrow \M{n,\Real}$ is a
diffeomorphism onto its image. Furthermore, the codomain of the $\log$
map is the set $\M{n,\Real}$, which is isomorphic to $\Real^{n^2}$, as
a vector space. Therefore, the $\log$ map is a local coordinate
transformation on $\GL{n,\Real}$, defined on the ball $B(I_n,1)$.

We denote by $e \in \M{n,\Real}$ the $\log$ coordinates
of the matrix $E \in B(I_n,1)$
\begin{equation}
\begin{aligned}
e \coloneqq \log(E) .
\label{eq:log_coordinates}
\end{aligned}
\end{equation}
To express the differential equation~\eqref{eq:basic_eqn} in $\log$
coordinates we differentiate $e$ with respect to time, making use of
Lemma~\ref{lemma:deriv_of_product} and
Proposition~\ref{prop:inverse_def}
\begin{equation*}
\begin{aligned}
\dot{e} &= \frac{\D}{\D t} \log(E)  \\
&= \frac{\D}{\D t} \left[ \sum_{k=1}^{\infty} \frac{(-1)^{k+1}}{k}
(E-I_n)^k \right] \\
&= \sum_{k=1}^{\infty} \frac{(-1)^{k+1}}{k} \frac{\D}{\D t} \left[
(E-I_n)^k \right] \\
&= \sum_{k=1}^{\infty} \frac{(-1)^{k+1}}{k} \left[ k \dot{E}
(E-I_n)^{k-1} \right] \\
&= \sum_{k=0}^{\infty} (-1)^{k} (E-I_n)^{k} \dot{E}  \\
&= E^{-1} \dot{E} \\
&= -a_0 E^{-1} E \log(E) \\
&= -a_0 e .
\label{eq:basic_eqn_log1}
\end{aligned}
\end{equation*}
The above equation, rewritten $\dot{e} = -a_0 e$, is linear with $n^2$
eigenvalues located at $-a_0$. Thus, for any positive constant $a_0 >
0$, the point $e = 0_n$ is an exponentially stable equilibrium of
$\dot{e} = -a_0 e$.  Since stability of an equilibrium is a coordinate
independent property, the equilibrium point $E = I_n$ is also
locally exponentially stable for~\eqref{eq:basic_eqn}.  The above
discussion proves the following.
\begin{lemma}
\label{lemma:basic_eqn}
On the set $E \in B(I_n,1)$, the vector field~\eqref{eq:basic_eqn} is
differentially equivalent to the vector field
\begin{equation}
\begin{aligned}
\dot{e} &= -a_0 e ,
\label{eq:basic_eqn_log}
\end{aligned}
\end{equation}
where $e \coloneqq \log(E)$. If $a_0 > 0$, the equilibrium point $E =
I_n$ of~\eqref{eq:basic_eqn} is locally exponentially stable.
\end{lemma}

The solution, $E(t)$, of the differential
equation~\eqref{eq:basic_eqn} can be expressed in closed form. This is
useful in obtaining an intuitive understanding of the
equation~\eqref{eq:basic_eqn}, but is not necessary for our main
argument.
\begin{proposition}
\label{prop:basic_eqn_soln}
Let $E_0 \in B(I_n,1)$ be sufficiently close to $I_n$ so that
$\log(E_0) \in B(0_n, \log(2))$ and let $a_ 0 > 0$ be arbitrary. Then
the solution of~\eqref{eq:basic_eqn} with initial condition $E(0) =
E_0$ is defined for all $t \geq 0$ and is given by
%
%
\begin{equation}
\begin{aligned}
E(t) &= \exp(\exp(-a_0 t) \log(E_0)) .
\label{eq:basic_eqn_soln}
\end{aligned}
\end{equation}
%
%
\end{proposition}
\begin{proof}
  First, we will show that the candidate
  solution~\eqref{eq:basic_eqn_soln} is a solution to the differential
  equation~\eqref{eq:basic_eqn} with initial condition $E_0$. First,
  we check the initial condition. The value at $t=0$ of $E(t)$ is
  $E(0) = E_0$ as required.

  Next, we check that $E(t)$ satisfies
  \eqref{eq:basic_eqn}. Differentiate $E(t)$ with respect to time
\begin{equation*}
\begin{aligned}
\frac{\D}{\D t} E(t) &= \frac{\D}{\D t} \left[
\exp(\exp(-a_0 t) \log(E_0)) \right] \\
&= -a_0 \exp(-a_0 t) \log(E_0) \exp(\exp(-a_0 t) \log(E_0)) \\
&= -a_0 \log(E) E ,
\label{eq:deriv_of_at}
\end{aligned}
\end{equation*}
where we have used the identity
\[
\log(E) = \log(\exp(\exp(-a_0 t) \log(E_0))) = \exp(-a_0 t) \log(E_0),
\]
which follows from Lemma \ref{lemma:explog_properties} (b) and the
assumption that $\log(E_0) \in B(0_n, \log(2))$.

To finish the proof, let us check that the solution $E(t)$
of~\eqref{eq:basic_eqn}, with initial condition $E_0$ as stated in the
proposition, is such that for all future times $t>0$, we have $E(t)
\in B(I_n,1)$. The solution of the linear differential
equation~\eqref{eq:basic_eqn_log} is $e(t) = e(0) \exp(-a_0 t)$, where
$e(0) = \log(E_0)$. Using the fact that $E(t) = \exp(e(t))$, we
compute upper bound on $\| E(t) - I_n \|$, for $t \geq 0$:
\begin{equation*}
\begin{aligned}
\| E(t) - I_n \| &= \left\| \sum_{k=0}^{\infty} \frac{1}{k!} e(0)^k \exp(-a_0 t k)
- I_n \right\| \\
&= \left\| \sum_{k=1}^{\infty} \frac{1}{k!} e(0)^k \exp(-a_0 t k)  \right\| \\
&\leq \sum_{k=1}^{\infty} \frac{1}{k!} \left\| e(0) \right\|^k \exp(-a_0 t k) \\
&\leq \sum_{k=1}^{\infty} \frac{1}{k!} \left\| e(0) \right\|^k \\
&= \exp(\left\| e(0) \right\|) - 1 \\
&< 2 - 1 = 1 .
\label{eq:ball_posinv}
\end{aligned}
\end{equation*}
The last inequality follows from the assumption that $\log(E_0) \in
B(0_n, \log(2))$. Thus, $E(t) \in B(I_n,1)$ for all $t \geq 0$.
\end{proof}
Let $E(t)$ be a solution of~\eqref{eq:basic_eqn}, which is initialized
at $E_0 = E(0)$, such that the conditions of
Proposition~\ref{prop:basic_eqn_soln} are satisfied.  Then, for all $t
\geq 0$, $E(t)$ stays on the same one-parameter subgroup, on which it
was initialized at time $t=0$. Indeed, by
Proposition~\ref{prop:basic_eqn_soln}, we have
\[
\left( \forall t \geq 0 \right) \; 
\left( \exists \tau \in \Real \right) \; \hspace{3mm} 
E(t) = \exp(\tau \log(E_0)) .
\]
Thus, in a neighbourhood of $I_n \in \GL{n,\Real}$, the vector
field~\eqref{eq:basic_eqn} is a linear vector field on Lie group,
since its flow is a one-parameter
subgroup~\cite{Jou:09},~\cite{Jou:11}.

\subsection{A Differential Equation on $\GL{n, \Real}$ and $\M{n,\Real}$}
\label{sec:LPSO_de_matrix}
Consider the following differential equation, which is a natural
extension of the differential equation~\eqref{eq:basic_eqn},
\begin{equation}
\begin{aligned}
\dot{E} &= e_2 E - a_{d-1} E \log(E) \\
\dot{e}_2 &= e_3 - a_{d-2} \log(E) \\
\vdots \\
\dot{e}_{d-1} &= e_d - a_1 \log(E)  \\
\dot{e}_d &= - a_0 \log(E) , 
\label{eq:LPSO_basic_eqn}
\end{aligned}
\end{equation}
where $E \in \GL{n,\Real}$, $e_i \in \M{n,\Real}$ for $i = 2, \hdots,
d$ and $a_0, \hdots, a_{d-1} \in \Real$ are constants such that the
polynomial $p(s) = s^d + a_{d-1} s^{d-1} + \cdots + a_1 s + a_0$ is
Hurwitz.  System~\eqref{eq:LPSO_basic_eqn} arises in the analysis of
the error dynamics associated with the direct
LPSO~\eqref{eq:LPSO_obs_direct1}.  

\begin{remark}
  Let $\g \subseteq \GL{n,\Real}$ be any linear Lie group, then the
  embedded submanifold, $S \coloneqq \g \times \Lie(\g) \times \cdots
  \times \Lie(\g)$, in the state space of~\eqref{eq:LPSO_basic_eqn} is
  positively invariant under the
  dynamics~\eqref{eq:LPSO_basic_eqn}. To see this, we check that, if
  $p = (E, e_2, \hdots, e_d) \in S$, then the vector
  field~\eqref{eq:LPSO_basic_eqn}, evaluated at $p$, lies in the
  tangent space to $S$ at $p$.

  Indeed, by Proposition~\ref{prop:lie_tangent_space}, we have that $E
  \log(E) \in T_E G$ and that $e_2 E \in T_E G$.  Therefore, the
  vector $\dot{E} = e_2 E - a_{d-1} E \log(E)$ is in the tangent space
  to $\g$ at $E$, i.e., $\dot{E} \in T_E \g$.  Furthermore we have
  that, for $i=2,\hdots,d$, $\dot{e}_i \in \Lie(\g) \simeq T_{e_i}
  \Lie(\g)$, because $\Lie(\g)$ is a vector space.  So the following
  holds
\begin{equation*}
\begin{aligned}
E \in G, \hspace{3mm} e_2 \in &\Lie(G), \hspace{3mm}
\hdots, \hspace{3mm} e_d \in \Lie(G)  \\
&\Downarrow \\
\dot{E} \in T_{E}G,  \hspace{3mm}
\dot{e}_2 \in &T_{e_2}\Lie(G), \hspace{3mm} \hdots, 
\hspace{3mm} \dot{e}_d \in T_{e_d}\Lie(G)  , 
\end{aligned}
\end{equation*}
Therefore, the vector field~\eqref{eq:LPSO_basic_eqn} is 
tangent to the submanifold $S$.  
\end{remark}

In general the matrices $E$ and $\dot{E}$ in~\eqref{eq:LPSO_basic_eqn}
do not commute. This is because $E$ and $e_2$ are generally
non-commuting matrices, i.e.,
\begin{equation*}
\begin{aligned}
\left[ E , \dot{E} \right] &= \left[ E , e_2 E - a_1 E \log(E) \right] \\
&= \left[ E , e_2 E \right] \\
&= E e_2 E - e_2 E^2 \\
&= \left[E , e_2 \right] E .
\end{aligned}
\end{equation*}
The non-commutativity of $E$ and $\dot{E}$ means that, defining $e_1
\coloneqq \log(E)$, the expression for $\dot{e}_1$ is not as simple as
was the case for equation~\eqref{eq:basic_eqn} in
Section~\ref{sec:LFSO_de_matrix}. In particular, we do not obtain a
closed-form expression for $\dot{e}_1$. Instead we have the following,
weaker, result.

%

\begin{proposition}
  In the open \nbhd $B(I_n, 1) \times \left(\M{n,\Real}\right)^{d-1} $
  the differential equation~\eqref{eq:LPSO_basic_eqn} is
  differentially equivalent to
\begin{equation*}
\begin{aligned}
\dot{e}_1 &= e_2 - a_{d-1} e_1 + K(e_1, e_2) \\
\dot{e}_2 &= e_3 - a_{d-2} e_1  \\
\vdots \\
\dot{e}_{d-1} &= e_d - a_1 e_1 \\
\dot{e}_d &= -a_0 e_1 , 
\end{aligned}
\end{equation*}
where $e_1 \coloneqq \log(E)$ and $K: \Lie(G) \times \Lie(G) \rightarrow
\Lie(G)$ is a smooth function that vanishes if $e_1$ and $e_2$
commute, i.e.,
\[
\left(\forall \; e_1, e_2 \in \M{n,\Real} \; [e_1, e_2] = 0_n\right)  \;
K(e_1, e_2) = 0_n.
\]
\end{proposition}
\begin{proof}
  Since $e_1 = \log(E)$ is a Taylor series in $E$, term by term
  differentiation yields that $\dot{e}_1$ only depends on $E$ and
  $\dot{E}$.  Furthermore, from~\eqref{eq:LPSO_basic_eqn}, we know
  that $\dot{E}$ only depends on $E$ and $e_2$.  Thus, using $E =
  \exp(e_1)$, we have that $\dot{e}_1$ only depends on $e_1$ and
  $e_2$.  Let $K(e_1, e_2) \coloneqq \dot{e}_1 - e_2 + a_{d-1} e_1$.

  Assume that $e_1$ and $e_2$ commute. This implies that $E =
  \exp(e_1)$ and $e_2$ also commute and this implies that $E$ and
  $\dot{E}$ commute. Since $E \dot{E} = \dot{E} E$, we can repeat
  almost the same analysis that we used in
  Section~\ref{sec:LFSO_de_matrix}, doing this we get

\begin{equation*}
\begin{aligned}
\dot{e}_1 &= \dot{E} E^{-1} \\
&= e_2 - a_{d-1} e_1 ,
\end{aligned}
\end{equation*}
therefore $K(e_1, e_2) = 0_n$ for any commuting $e_1$ and $e_2$.

The expressions of $\dot{e}_i$ for $i = 2, \hdots, d$ are computed by
substituting $\log(E) = e_1$ into~\eqref{eq:LPSO_basic_eqn}.
\end{proof}

\begin{lemma}
\label{lemma:LPSO_basic_eqn_stable}
If the constants $a_0, \hdots, a_{d-1} \in \Real$ are chosen such that
the polynomial $p(s) = s^d + a_{d-1} s^{d-1} + \cdots + a_1 s + a_0$
is Hurwitz then the equilibrium point $(E, e_2, \ldots, e_d) = (I_n,
0_n, \ldots, 0_n)$ of the differential equation~\eqref{eq:LPSO_basic_eqn}
is locally exponentially stable.
\end{lemma}
\begin{proof}
  Adapting the proof of~\cite[Theorem 3.1 (ii)]{MalHamMahMor:09}, we
  show that~\eqref{eq:LPSO_basic_eqn} is locally exponentially stable
  at the equilibrium point, by showing that its linearization, around
  the equilibrium point $(I_n, 0_n, \ldots, 0_n)$, is exponentially
  stable.

  In a \nbhd of the equilibrium point $(I_n, 0_n, \ldots, 0_n)$ define
\[
\delta E \coloneqq E - I_n, \hspace{7mm} 
\delta e_2 \coloneqq e_2 - 0_n, \hspace{7mm} 
\hdots, \hspace{7mm} 
\delta e_d \coloneqq e_d - 0_n .
\]
Using the series definition of the matrix
logarithm~\eqref{eq:logarithm_def}
\[
\log(E) = \delta E - \frac{1}{2} (\delta E)^2 + \cdots
\]
we deduce that, near $\delta E = 0_n$,
\[
\log(E) \approx \delta E.
\]
Similarly, using $E = \delta E + I_n$, 
and dropping higher order terms in $\delta E$, we get 
\[
E \log(E) = (\delta E + I_n) 
\left( (\delta E) - \frac{1}{2} (\delta E)^2 + \cdots \right) 
\approx \delta E .
\]
Finally, near the equilibrium point $(I_n, 0_n, \ldots, 0_n)$,
\[
e_2 E = \left(\delta e_2 \right) 
\left(\delta E + I_n \right) \approx \delta e_2 . 
\]
Substituting these approximations into the differential
equation~\eqref{eq:LPSO_basic_eqn}, we get the linearization
of~\eqref{eq:LPSO_basic_eqn} at $(I_n, 0_n, \ldots, 0_n)$
\begin{equation*}
\begin{aligned}
\frac{\D}{\D t} \left[ \begin{array}{c} 
\delta E \\
\delta e_2 \\
\delta e_3 \\ 
\vdots \\
\delta e_{d-1} \\
\delta e_d 
\end{array} \right] = \left( \begin{array}{cccccc}
-a_{d-1} I_n & I_n & 0_n & \hdots & 0_n & 0_n \\
-a_{d-2} I_n & 0_n & I_n & \hdots & 0_n & 0_n \\
-a_{d-3} I_n & 0_n & 0_n & \hdots & 0_n & 0_n \\
\vdots & \vdots & \vdots & \ddots & \vdots & \vdots \\
-a_{1} I_n & 0_n & 0_n & \hdots & 0_n & I_n \\
-a_{0} I_n & 0_n & 0_n & \hdots & 0_n & 0_n 
\end{array} \right) \left[ \begin{array}{c} 
\delta E \\
\delta e_2 \\
\delta e_3 \\ 
\vdots \\
\delta e_{d-1} \\
\delta e_d 
\end{array} \right]  .
\end{aligned}
\end{equation*}
The eigenvalues of the system matrix are located at the roots of the
polynomial $p(s) = s^d + a_{d-1} s^{d-1} + \cdots + a_1 s + a_0$, with
multiplicity $n$, for each (possibly repeating) root of $p(s)$.  Since
all the eigenvalues have negative real parts, the linearization above
is exponentially stable. Therefore $(E, e_2, \hdots, e_d) = (I_n, 0_n,
\hdots, 0_n)$ is a locally exponentially stable equilibrium
of~\eqref{eq:LPSO_basic_eqn}.
\end{proof}

\section{Estimation Error Dynamics}
\label{sec:errordyn}

In this section we analyse the stability of the estimation error for
the each of the observers proposed in Section~\ref{sec:proposed}. We
show that, under Assumptions~\ref{ass:bounded}
and~\ref{ass:dyn_bounded}, the estimates exponentially
converge to the state of the system.

\subsection{Local full state observers}
\label{sec:LFSO_observer_proof}
We first analyze the dynamics of the error functions $E_l$ and $E_r$
under the observers defined by~\eqref{eq:obs_passive1}
and~\eqref{eq:obs_direct1}. Our analysis makes frequent use of Lemma
\ref{lemma:inverse_deriv}.  We assume that $\hat{X}$ is initialized
sufficiently close to $X$, so that $E_l$, $E_r \in B(I_n,1)$. This
assumption is sufficient to ensure that the series definitions,
using~\eqref{eq:logarithm_def}, of $\log(E_r)$ and $\log(E_l)$ are
convergent.

\subsubsection{Passive Observer}
When the passive observer~\eqref{eq:obs_passive1} is used to estimate
the state of~\eqref{eq:leftinv_system_def} dynamics of the
right-invariant error, $E_r$, making use of
Lemma~\ref{lemma:inverse_deriv}, are
\begin{equation}
\begin{aligned}
\dot{E}_r &= \frac{d}{dt} \left[ \hat{X} X^{-1} \right] \\
&= \dot{\hat{X}} X^{-1} - \hat{X} X^{-1} \dot{X} X^{-1} \\
&= \hat{X} u X^{-1} - a_0 \hat{X} \log(X^{-1} \hat{X}) X^{-1} -
\hat{X} u X^{-1} \\
&= -a_0 \hat{X} \log(X^{-1} \hat{X}) X^{-1} \\
&= -a_0 \hat{X} X^{-1} X \log(X^{-1} \hat{X}) X^{-1} \\
&= -a_0 \hat{X} X^{-1} \log(\hat{X} X^{-1}) \\
&= -a_0 E_r \log(E_r) .
\label{eq:rightinv_error_passive_obs}
\end{aligned}
\end{equation}
The above differential equation is formally the same as
equation~\eqref{eq:basic_eqn}. Therefore if $\hat{X}$ is sufficiently
close to $X$ so that $E_r \in B(I_n,1)$ then, by
Lemma~\ref{lemma:basic_eqn},
system~\eqref{eq:rightinv_error_passive_obs} is differentially
equivalent to
\begin{equation}
\begin{aligned}
\dot{e}_r = -a_0 e_r , 
\label{eq:rightinv_error_passive_obs_log}.
\end{aligned}
\end{equation}
By choosing $a_0 > 0$, Lemma~\ref{lemma:basic_eqn} states the
equilibrium point $E_r = I_n$ is locally exponentially stable for
system~\eqref{eq:rightinv_error_passive_obs}. This discussion, in light
of Proposition~\ref{prop:norm_convergence}, proves the following solution
to Problem~\eqref{prob:kinematic}.

\begin{corollary}
\label{cor:prob1}
For $E_r(0) \in B(I_n, 1)$, the passive
observer~\eqref{eq:obs_passive1} exponentially stabilizes $E_r =
I_n$. Furthermore, under Assumption~\ref{ass:bounded}, the passive
observer solves Problem~\eqref{prob:kinematic}.
\end{corollary}

The convergence of $E_r$ to $I_n$ does not rely on the trajectories
of~\eqref{eq:leftinv_system_def} being bounded. Next, we examine the
dynamics of the left-invariant error, $E_l$, to see if
Assumption~\ref{ass:bounded} can be weakened. The dynamics of the left
invariant error $E_l$ under the passive
observer~\eqref{eq:obs_passive1} are
\begin{equation}
\begin{aligned}
\dot{E}_l &= \frac{d}{dt} \left[ X^{-1} \hat{X} \right] \\
&= - X^{-1} \dot{X} X^{-1} \hat{X} + X^{-1} \dot{\hat{X}} \\
&= - u X^{-1} \hat{X} + X^{-1} \hat{X} u - a_0 X^{-1} \hat{X}
\log(X^{-1} \hat{X}) \\
&= -u E_l + E_l u - a_0 E_l \log(E_l) \\
&= - a_0 E_l \log(E_l)  + \delta_P(u, E_l), 
\label{eq:leftinv_error_passive_obs}
\end{aligned}
\end{equation}
where $\delta_P(u, E_l) \coloneqq E_l u - u E_l$ is a perturbation
term that vanishes when $E_l = I_n$.  Since the matrices $E_l$ and
$\dot{E}_l$ do not, in general, commute,
Lemma~\ref{lemma:deriv_of_product} does not hold
for~\eqref{eq:leftinv_error_passive_obs}.


Next we transform the error
dynamics~\eqref{eq:leftinv_error_passive_obs} into $\log$
coordinates. Recall, by Lemma~\ref{lemma:log_error_relationship}, if
$\hat{X}$ is sufficiently close to $X$ then $e_l = X^{-1} e_r
X$. Therefore to transform the dynamics
\eqref{eq:leftinv_error_passive_obs} into $e_l$ coordinates, we just
differentiate this alternate expression for $e_l$
\begin{equation}
\begin{aligned}
\dot{e}_l &= \frac{d}{dt} \left[ X^{-1} e_r X \right] \\
&= -X^{-1} \dot{X} X^{-1} e_r X + X^{-1} \dot{e}_r X + X^{-1} e_r
\dot{X} \\
&= -u e_l - a_0 e_l + e_l u \\
&=  - a_0 e_l + \left[ e_l, u \right] .
\label{eq:leftinv_error_passive_obs_log}
\end{aligned}
\end{equation}
The above system, rewritten $\dot{e}_l = -a_0 e_l + \left[ e_l, u
\right]$, is bilinear. If $a_0 < 0$, then by~\cite[Corollary
4]{sontag1998comments},
system~\eqref{eq:leftinv_error_passive_obs_log} is integral-input to
state stable (iISS). Specifically, see~\cite{sontag1998comments},
there exist class-$\mathcal{K}_\infty$ functions $\alpha$, $\gamma$
and a class-$\mathcal{KL}$ function $\beta$ such that for any $e_l(0)
\in \M{n, \Real}$, and any input $u(\cdot)$
\[
\alpha(\|e_l(t)\|) \leq \beta(e_l(0), t) + \int_0^t\gamma(\|u(\tau)\|)\D\tau.
\]
As a result, if $u(t) \to 0_n$ as $t \to \infty$, then $e_l(t) \to 0_n$
as $t \to \infty$. Furthermore, if $\int^\infty_0\gamma(\|u(t)\|)\D t
< \infty$, then by~\cite[Proposition 6]{sontag1998comments}, $e_l(t)
\to 0_n$ as $t \to \infty$. Niether of these properties allow us to
weaken Assumption~\ref{ass:bounded}. First, because we have no
guarantees that the control signal satisfies the above properties and
second, System~\eqref{eq:leftinv_error_passive_obs} is only
differentially equivalent to~\eqref{eq:leftinv_error_passive_obs_log}
if $E_l \in B(I_n, 1)$ and the iISS property does not ensure that $e_l
\in \log{(B(I_n, 1))}$.

By showing that the system~\eqref{eq:leftinv_error_passive_obs} is
diffeomorphic to the system~\eqref{eq:leftinv_error_passive_obs_log},
we have found an easy way to prove the following, non-obvious, result.

\begin{corollary}
\label{corr:commutator_log}
Let $G \subseteq \GL{n,\Real}$ be a linear Lie group and consider the
system
\begin{equation}
\begin{aligned}
\dot{E} = [E, u]  , 
\label{eq:nonbasic_eqn_group}
\end{aligned}
\end{equation}
where $E \in \g \subseteq \GL{n, \Real}$ is the state and $u \in
\Lie(G) \subseteq \M{n,\Real}$ is an admissible input signal.  On the
open set $B(I_n,1) \cap \g$, system~\eqref{eq:nonbasic_eqn_group} is
differentially equivalent to
\begin{equation}
\begin{aligned}
\dot{e} = [e,u] ,
\label{eq:nonbasic_eqn_algebra}
\end{aligned}
\end{equation}
where $e = \log(E)$.
\end{corollary}

\begin{proof}
Rewrite~\eqref{eq:nonbasic_eqn_group} as a difference of two vector fields 
\begin{equation*}
\begin{aligned}
\dot{E} &= \left( [E, u] + E \log(E) \right) - \left( E \log(E) \right) \\ 
&= f(E,u) - g(E) , 
\end{aligned}
\end{equation*}
where $f(E,u) \coloneqq [E, u] + E \log(E)$ 
and $g(E) \coloneqq E \log(E)$.
Since the system~\eqref{eq:leftinv_error_passive_obs}  
transforms into the system~\eqref{eq:leftinv_error_passive_obs_log}, 
we know that the vector field $f(E,u)$ transforms into 
$[e,u] + e$. Also, since the dynamics~\eqref{eq:basic_eqn} transform 
into the dynamics~\eqref{eq:basic_eqn_log}, we know that 
the vector field $g(E)$ transforms into $e$. 
This means that the vector field $f(E,u) - g(E)$ transforms into 
$[e,u] + e - e = [e, u]$.
\end{proof}

Recall, that in equation~\eqref{eq:basic_eqn}, we were able to easily
differentiate the Taylor series expansion of $\log(E)$, because the
matrices $E$ and $\dot{E} = -a_0 E \log(E)$ commute, i.e., $E \dot{E}
= \dot{E} E$.  However, in the equation~\eqref{eq:nonbasic_eqn_group},
the matrices $E$ and $\dot{E} = [E,u]$ do not commute, i.e., in
general $E \dot{E} \ne \dot{E} E$.  This non-commutativity makes it
very difficult to differentiate the Taylor series expansion of
$\log(E)$, when $\dot{E} = [E,u]$, as
in~\eqref{eq:nonbasic_eqn_group}.  Thus, it seems that the result of
Corollary~\ref{corr:commutator_log} is difficult to directly obtain by
differentiating the series expansion of $\log(E)$ and substituting
$\dot{E} = [E, u]$. Our analysis of
equation~\eqref{eq:nonbasic_eqn_group} is facilitated by taking the
systemic view of ``splitting" the
equation~\eqref{eq:nonbasic_eqn_group} into a pair consisting of
``system"~\eqref{eq:leftinv_system_def}, with state $X$, and
``observer"~\eqref{eq:obs_passive1}, with state $\hat{X}$.  The
splitting is done as $E = X^{-1} \hat{X}$, and allows us to convert
the differential equation~\eqref{eq:nonbasic_eqn_group} into $\log$
coordinates.

\subsubsection{Direct Observer}

When the direct observer~\eqref{eq:obs_direct1} is used to estimate
the state of~\eqref{eq:leftinv_system_def} dynamics of the
left-invariant error, $E_l$, making use of
Lemma~\ref{lemma:inverse_deriv}, are
\begin{equation}
\begin{aligned}
\dot{E}_l &=  \frac{d}{dt} \left[ X^{-1} \hat{X} \right] \\
&= - X^{-1} \dot{X} X^{-1} \hat{X} + X^{-1} \dot{\hat{X}} \\
&= - u X^{-1} \hat{X} + u X^{-1} \hat{X} - a_0 X^{-1}
\hat{X} \log(X^{-1} \hat{X}) \\
&= -u E_l + u E_l - a_0 E_l \log(E_l) \\
&= -a_0 E_l \log(E_l) .
\label{eq:leftinv_error_direct_obs}
\end{aligned}
\end{equation}
The above equation~\eqref{eq:leftinv_error_direct_obs} is the same as
the equation~\eqref{eq:basic_eqn}, if we identify $E_l$ with $E$.
This means that if $\hat{X}$ is sufficiently close to $X$ so that $E_l
\in B(I_n,1)$, then by Lemma~\ref{lemma:basic_eqn},
system~\eqref{eq:leftinv_error_direct_obs} in $e_l$-coordinates reads
\begin{equation}
\begin{aligned}
\dot{e}_l = -a_0 e_l.
\label{eq:leftinv_error_direct_obs_log}
\end{aligned}
\end{equation}
If $a_0 > 0$, Lemma~\ref{lemma:basic_eqn} states that the equilibrium
point $E_l = I_n$ is locally exponentially stable for the
dynamics~\eqref{eq:leftinv_error_direct_obs}.

\begin{corollary}
\label{cor:prob2}
For $E_l(0) \in B(I_n, 1)$, the direct
observer~\eqref{eq:obs_passive1} exponentially stabilizes $E_l =
I_n$. Furthermore, under Assumption~\ref{ass:bounded}, the passive
observer solves Problem~\eqref{prob:kinematic}.
\end{corollary}

As before, we seek to weaken Assumption~\ref{ass:bounded} and hence we
examine the dynamics of the right-invariant error $E_r$, when the
direct observer is used
\begin{equation}
\begin{aligned}
\dot{E}_r &= \frac{d}{dt} \left[ \hat{X} X^{-1} \right] \\
&= \dot{\hat{X}} X^{-1} - \hat{X} X^{-1} \dot{X} X^{-1} \\
&= X u X^{-1} \hat{X} X^{-1} - a_0 \hat{X} \log(X^{-1} \hat{X})
X^{-1} - \hat{X} u X^{-1} \\
&= X u X^{-1} E_r - E_r X u X^{-1} - a_0 E_r X \log(E_l) X^{-1} \\
&= X u X^{-1} E_r - E_r X u X^{-1} - a_0 E_r \log(E_r) \\
&= \delta_D(u, X, E_r) - a_0 E_r \log(E_r). 
\label{eq:rightinv_error_direct_obs}
\end{aligned}
\end{equation}
Here, $\delta_D(u, X, E_r) \coloneqq X u X^{-1} E_r - E_r X u X^{-1}$
is a perturbation term that vanishes when $E_r = I_n$.  The above
equation~\eqref{eq:rightinv_error_direct_obs} has the same problem
that we encountered when trying to analyze
equation~\eqref{eq:leftinv_error_passive_obs}.  Namely, the matrices
$E_r$ and $\dot{E}_r$ do not commute in general, because $E_r$ and
$\delta_D(u, X, E_r)$ do not commute in general.  Fortunately, we can
transform equation~\eqref{eq:leftinv_error_passive_obs} into $\log$
coordinates by once again differentiating the identity $e_r = X e_l
X^{-1}$. To be able to do this, it is sufficient that the conditions
of Lemma~\ref{lemma:log_error_relationship} are satisfied, i.e., that
$E_l, E_r \in B(I_n,1)$. Doing so, one obtains
\begin{equation}
\begin{aligned}
\dot{e}_r &= \frac{d}{dt} \left[ X e_l X^{-1} \right] \\
&= \dot{X} e_l X^{-1} + X \dot{e}_l X^{-1} - X e_l X^{-1} \dot{X}
X^{-1} \\
&= X u e_l X^{-1} - a_0 X e_l X^{-1} - X e_l u X^{-1} \\
&= -a_0 e_r + \left[ X u X^{-1}, e_r \right] .
\label{eq:rightinv_error_direct_obs_log}
\end{aligned}
\end{equation}
The above system, rewritten $\dot{e}_r = -a_0 e_r + \left[ X u X^{-1},
  e_r \right]$, is a non-autonomous, bilinear system. Once again, we
cannot weaken the requirement of Assumption~\ref{ass:bounded} and rely
on Proposition~\ref{prop:norm_convergence} to ensure that $E_l \to 0_n$
as $t \to \infty$ is equivalent to $E_r \to 0_n$ as $t \to \infty$.

\subsection{Local partial state observers}
\label{sec:LPSO_observer_proof}
We now analyze the estimation error dynamics when using the observers
proposed in Section~\ref{sec:LPSO_observer} and defined
by~\eqref{eq:LPSO_obs_passive1} and~\eqref{eq:LPSO_obs_direct1}.

\subsubsection{Direct Observer}
When the direct observer~\eqref{eq:LPSO_obs_direct1} is applied to
estimate the state of system~\eqref{eq:LPSO_system_def} the dynamics
of the right-invariant error $E_l$ are
\begin{equation}
\begin{aligned}
\dot{E}_l &= e_2 E_l  - a_{d-1} E_l \log( E_l ) \\
\dot{e}_2 &= e_{3} - a_{d-2} \log(E_l) \\
\vdots \\
\dot{e}_{d-1} &= e_d - a_1 \log(E_l) \\
\dot{e}_d &= - a_{0} \log(E_l) .
\label{eq:LPSO_all_error_direct_obs}
\end{aligned}
\end{equation}
The above differential equation is formally the same as
equation~\eqref{eq:LPSO_basic_eqn}, if we identify $E$ with
$E_l$. Application of Lemma~\ref{lemma:LPSO_basic_eqn_stable}
immediately yields the following solution to
Problem~\ref{prob:dynamic}.

\begin{corollary}
\label{cor:P2}
For $(E_l, e_2, \hdots, e_d) \in B(I_n, I) \times \left(\M{n,
    \Real}\right)$, the direct observer~\eqref{eq:LPSO_obs_direct1}
exponentially stabilizes $(I_n, 0_n, \ldots, 0_n)$. Furthermore, under
Assumption~\ref{ass:dyn_bounded}, the direct observer solves
Problem~\ref{prob:dynamic}.
\end{corollary}

\subsubsection{Passive Observer}
When the passive observer~\eqref{eq:LPSO_obs_passive1} is employed to
estimate the state of system~\eqref{eq:LPSO_system_def} the
dynamics of the right-invariant error $E_r$ are given by
\begin{equation}
\begin{aligned}
\dot{E}_r &= E_r \Ad_X ( e_2 ) 
 - a_{d-1} E_r \log( E_r ) \\
\dot{e}_2 &= e_{3} - a_{d-2} \log(E_r) \\
\vdots \\
\dot{e}_{d-1} &= e_d - a_1 \log(E_r) \\
\dot{e}_d &= - a_{0} \log(E_r) .
\label{eq:LPSO_all_error_passive_obs}
\end{aligned}
\end{equation}
Lemma~\ref{lemma:LPSO_basic_eqn_stable} cannot be used to deduce the
stability of the equilibrium point $(E_r, e_2, \hdots, e_d) = (I_n,
0_n, \hdots, 0_n)$.  Unfortunately, we are not able to prove the
stability of these error dynamics. We conjecture that the passive LPSO
is locally exponentially convergent if
Assumption~\ref{ass:dyn_bounded} holds.  This conjecture is
supported by simulation, where the passive LPSO performs better than
the direct LPSO, when a large amount of measurement noise is present
in $Y$.


\section{Examples}
\label{sec:examples}

\subsection{State estimation on $\SO{3, \Real}$}
\label{sec:LFSO_SO3}
Recall the kinematic model of a rotating rigid body
~\eqref{eq:system_so3} introduced in Section~\ref{sec:motivation}. For
the kinematics~\eqref{eq:system_so3}, the proposed passive observer is
\begin{equation}
\begin{aligned}
\dot{\hat{R}} &= \hat{R} \omega + a_0 \hat{R} \log(Y^{\top} \hat{R}) ,
\label{eq:so3_passive_obs1}
\end{aligned}
\end{equation}
and the direct observer is
\begin{equation}
\begin{aligned}
\dot{\hat{R}} &= Y \omega Y^\top \hat{R} + a_0 \hat{R} \log(Y^\top
\hat{R}).
\label{eq:so3_direct_obs1}
\end{aligned}
\end{equation}
We now briefly discuss the connection between our two observers and
the filters from~\cite{MahHamPfl:05}, \cite{MahHamPfl:08}. The authors
propose passive and direct filters that are similar to the ones
proposed herein. The passive filter of \cite{MahHamPfl:05},
\cite{MahHamPfl:08} is
\begin{equation}
\begin{aligned}
\dot{\hat{R}} &= \hat{R} \omega + k \hat{R} \pi_a(\hat{R}^\top R) ,
\label{eq:so3_passive_obs2}
\end{aligned}
\end{equation}
and the direct filter of \cite{MahHamPfl:05}, \cite{MahHamPfl:08} is
\begin{equation}
\begin{aligned}
\dot{\hat{R}} &= \Ad_R(\omega) \hat{R} + k \Ad_{\hat{R}}(
\pi_a(\hat{R}^\top R) ) \hat{R} \\
&= R \omega R^\top \hat{R} + k \hat{R} \pi_a(\hat{R}^\top R) ,
\label{eq:so3_direct_obs2}
\end{aligned}
\end{equation}
where $k \in \Real$ is a positive constant that controls the rate of
observer convergence and $\pi_a : \M{n, \Real} \to \Lie(\SO{3,
  \Real})$, $A \mapsto \frac{1}{2}(A - A^T)$ is the anti-symmetric
projection operator.

The synchronization terms of the two passive observers
\eqref{eq:so3_passive_obs1} and \eqref{eq:so3_passive_obs2} are the
same: $\hat{R} \omega$. The synchronization terms of the two direct
observers \eqref{eq:so3_direct_obs1} and \eqref{eq:so3_direct_obs2}
are also the same: $R \omega R^T \hat{R}$. However, the innovation
term of our proposed observer, when applied to $\SO{3, \Real}$ is
different from the innovation term of the observers proposed in
\cite{MahHamPfl:05}, \cite{MahHamPfl:08}. In particular, for any $A
\in \SO{3, \Real}$ we can express $\log{(A)}$ in terms of the
anti-symmetric projection operator using~\cite[Section
III. C.]{MahHamPfl:05}
\begin{equation}
\begin{aligned}
\pi_a(A) = \frac{\sin(\theta_A)}{\theta_A} \log(A)
\label{eq:asymproj_iden_gen}
\end{aligned}
\end{equation}
where $\theta_A$ is the rotation angle corresponding to the axis-angle
representation of $A$. The angle $\theta_A$ can be
computed~\cite{SpoHutVid:06} as $\theta_{A} \coloneqq \arccos \left(
  \frac{\tr(A)- 1}{2} \right)$.

Using the identity~\eqref{eq:asymproj_iden_gen}, we can express the
observers~\eqref{eq:so3_passive_obs1} and~\eqref{eq:so3_direct_obs1}
in terms of the anti-symmetric projection operator, rather than the
logarithm map.  For $\theta_{Y^{\top} \hat{R}} \ne \pm \pi$, the
passive LFSO~\eqref{eq:so3_passive_obs1} is rewritten as
\begin{equation}
\begin{aligned}
\dot{\hat{R}} &= \hat{R} u - a_0 
\frac{\theta_{Y^{\top} \hat{R}}}{\sin(\theta_{Y^{\top} \hat{R}})} 
\hat{R} \pi_a(Y^{\top} \hat{R}) 
\label{eq:so3_passive_obs1_proj}
\end{aligned}
\end{equation}
and the direct LFSO~\eqref{eq:so3_direct_obs1} is rewritten as
\begin{equation}
\begin{aligned}
\dot{\hat{R}} &= Y u Y^{\top} \hat{R} - a_0 
\frac{\theta_{Y^{\top} \hat{R}}}{\sin(\theta_{Y^{\top} \hat{R}})} 
\hat{R} \pi_a(Y^{\top} \hat{R}) .
\label{eq:so3_direct_obs1_proj}
\end{aligned}
\end{equation}
If we take the observer gains to be equal $k = a_0$,
then~\eqref{eq:so3_passive_obs1_proj},~\eqref{eq:so3_direct_obs1_proj}
differ from~\eqref{eq:so3_passive_obs2},~\eqref{eq:so3_direct_obs2} by
the scalar quantity $\frac{\theta_{Y^{\top}
    \hat{R}}}{\sin(\theta_{Y^{\top} \hat{R}})}$, which appears in the
innovation terms
of~\eqref{eq:so3_passive_obs1_proj},~\eqref{eq:so3_direct_obs1_proj},
but does not appear in the innovation terms
of~\eqref{eq:so3_passive_obs2},~\eqref{eq:so3_direct_obs2}.  This
scalar quantity approaches $1$ as $\theta_{Y^{\top} \hat{R}}$
approaches $0$.  Thus, on $\SO{3, \Real}$ and for small
$\theta_{Y^{\top} \hat{R}}$, our observers behave similar to the
observers of~\cite{MahHamPfl:05},~\cite{MahHamPfl:08}.

We now simulate our direct and passive LFSOs. The initial conditions
for the plant and the observer are chosen as
\begin{equation*}
\begin{aligned}
R(0) = \left(
\begin{array}{ccc}
0.6330 & -0.1116 & -0.7660  \\
0.7128 & -0.3020 & 0.6330  \\
-0.3020 & -0.9467 & -0.1116 
\end{array}
\right) ,   \qquad \hat{R}(0) = I_3.
\end{aligned}
\end{equation*}
The angular velocity input is as
\begin{equation*}
\begin{aligned}
u(t) = \left(
\begin{array}{ccc}
0 & -2 \sin(t) & \cos(t)  \\
2 \sin(t) & 0 & -\sin(t)  \\
-\cos(t) & \sin(t) & 0 
\end{array}
\right)  .
\end{aligned}
\end{equation*}
The observer gain is taken to be $a_0 =
1$. Figure~\ref{fig:plot_noiseless} shows the simulation results when
there is no noise, i.e., when $Y$ is exactly equal to $R$.  From these
plots we see that, when there no measurement noise, the direct and
passive LFSOs have similar performance.

The Lie group estimation errors, $E_l$ and $E_r$, are initially (at
time $t=0$) $E_l(0) = R^{-1}(0) \hat{R}(0) = R^{-1}(0)$, and $E_r(0) =
\hat{R}(0) R^{-1}(0) = R^{-1}(0)$.  Calculating the distance between
$E_l(0)$ or $E_r(0)$ and $I_3$, we get $\| E_l(0) - I_3 \| = \| E_r(0)
- I_3 \| = 1.6675$. Therefore $E_l(0), E_r(0) \notin B(I_3,1)$.
Strictly speaking, our analysis of the Lie group estimation error
dynamics was performed on the ball $B(I_n,1)$. The simulation shows
that that the LFSOs are nevertheless convergent and suggests that the
region of convergence of the LFSO is larger than $B(I_n,1)$.

Figure~\ref{fig:plot_noise} shows simulations of our observers in the
presence of measurement noise. The noisy measurement $Y$ is obtained
by multiplying $R$ by a randomly generated rotation matrix $N$, to
obtain: $Y = R N$.  To generate the random rotation matrix $N \in
\SO{3, \Real}$, we generate a random skew-symmetric matrix, $n \in
\Lie(\SO{3})$, whose elements are normally distributed, with zero-mean
and standard deviation of $\sigma = 0.4$.  The random rotation matrix
$N \in \SO{3}$ is then computed as $N \eqdef \exp(n)$.

From Figure~\ref{fig:plot_noise}, we see that the passive observer
appears to be more robust with respect to noise than the direct
observer. This observation can be explained by noting that when $Y =
RN$, the synchronization term $R N\omega N^\top R^\top$ of the direct
observer is explicitly impacted by the noise. The noise has no
explicit effect on the synchronization term of the passive observer.
\begin{figure}[ht]
\centering
\includegraphics[width=0.5\textwidth]{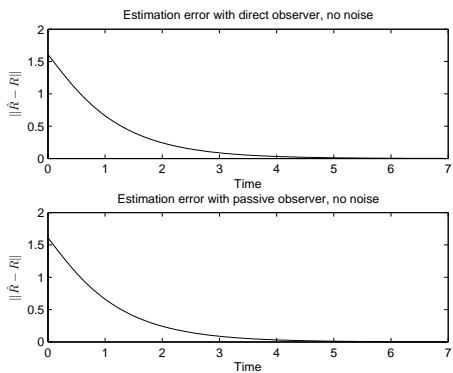}
\caption{$\| \hat{R} - R \|$ versus time for the proposed LFSO
  observers on $\SO{3, \Real}$ without measurement noise. Typical
  result from repeated testing.}
\label{fig:plot_noiseless}
\end{figure}
\begin{figure}[ht]
\centering
\includegraphics[width=0.5\textwidth]{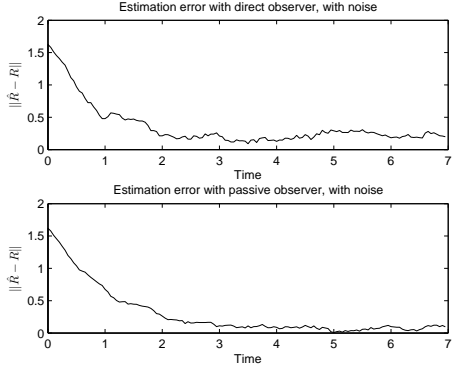}
\caption{$\| \hat{R} - R \|$ versus time for the proposed LFSO
  observers on $\SO{3, \Real}$ with a significant amount of
  measurement noise. Typical result from repeated testing.}
\label{fig:plot_noise}
\end{figure}

\subsection{Dynamic Rigid-Body Orientation Estimation on $\SO{3}$}
\label{sec:LPSO_SO3}
Recall the dynamic model of a rotating rigid body
~\eqref{eq:dynamic_so3} introduced in
Section~\ref{sec:motivation}. A similar model was discussed
in~\cite[Example 2]{Brockett:72}. For system~\eqref{eq:system_so3},
the proposed direct observer is
\begin{equation}
\begin{aligned}
\dot{\hat{R}} &= Y \hat{\omega} Y^{-1} \hat{R} - a_1 \hat{R} \log(Y^{-1} \hat{R}) \\
\dot{\hat{\omega}} &= u - a_0 \log(Y^{-1} \hat{R}) 
\label{eq:LPSO_so3_direct_obs}
\end{aligned}
\end{equation}
and the passive LPSO is
\begin{equation}
\begin{aligned}
\dot{\hat{R}} &= \hat{R} \hat{\omega} - a_1 \hat{R} \log(Y^{-1} \hat{R}) \\
\dot{\hat{\omega}} &= u - a_0 \log(Y^{-1} \hat{R}) . 
\label{eq:LPSO_so3_passive_obs}
\end{aligned}
\end{equation}
We simulate the direct and the passive LPSOs, with increasing amounts
of noise in the output.  The initial conditions for the plant and the
observer are chosen as
\begin{equation*}
\begin{aligned}
R(0) = \left(
\begin{array}{ccc}
0 & 1 & 0  \\
0 & 0 & 1  \\
1 & 0 & 0 
\end{array}
\right) ,   \hspace{3mm}  
\omega(0) = \left(
\begin{array}{ccc}
0 & -1 & 1  \\
1 & 0 & -1  \\
-1 & 1 & 0 
\end{array}
\right)
,   \hspace{3mm}  
\hat{R}(0) = I_3
,   \hspace{3mm}  
\hat{\omega}(0) = 0_3 
 .
\end{aligned}
\end{equation*}
The angular acceleration input is chosen to be
\begin{equation*}
\begin{aligned}
u(t) = \left(
\begin{array}{ccc}
0 & -2 \sin(t) & \cos(t)  \\
2 \sin(t) & 0 & -\sin(t)  \\
-\cos(t) & \sin(t) & 0 
\end{array}
\right)  .
\end{aligned}
\end{equation*}
The observer gains are chosen as 
\[
a_0 = 1, \hspace{7mm} a_1 = 2  .
\]
Noise is injected into the output via the random rotation matrix $N
\in \SO{3, \Real}$, by setting $Y = R N$.  The matrix $N$ is generated
as in Section~\ref{sec:LFSO_SO3}.  The simulation results are shown in
Figure~\ref{fig:so3_sim_05}.  From the simulations, when there is no
measurement noise, the direct LPSO appears to converge faster than the
passive LPSO.  However, with a large amount of measurement noise, the
passive LPSO appears to be comparable to the direct LPSO.
%
%
%
%
\begin{figure}[ht!]
\centering
\begin{subfigure}{0.5\textwidth}
\centering
\includegraphics[width=\textwidth]{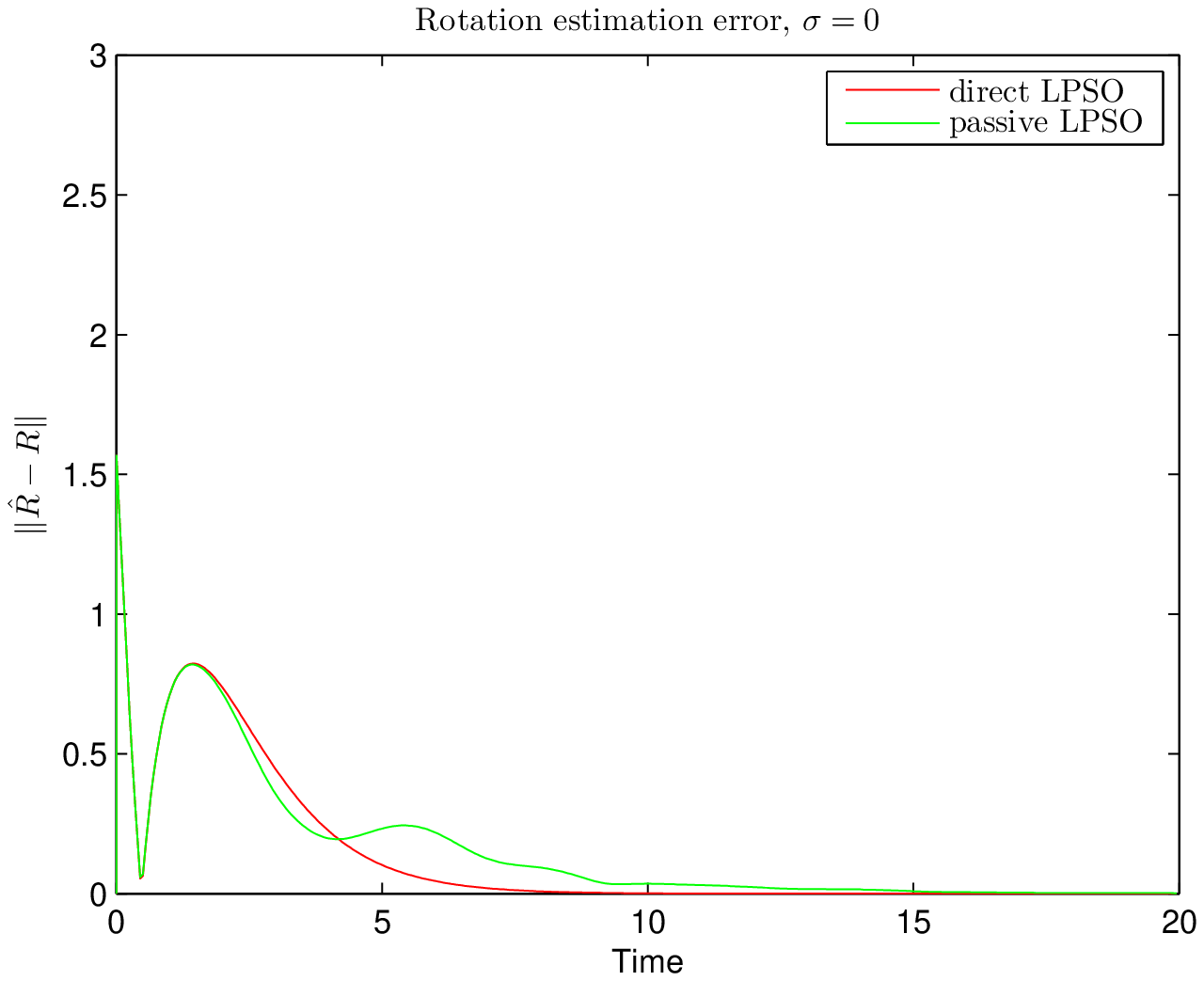}
\caption{$\|\hat{R} - R\|$ versus time with $\sigma = 0$.}
\label{fig:gull}
\end{subfigure}%
~
\begin{subfigure}{0.5\textwidth}
\centering
\includegraphics[width=\textwidth]{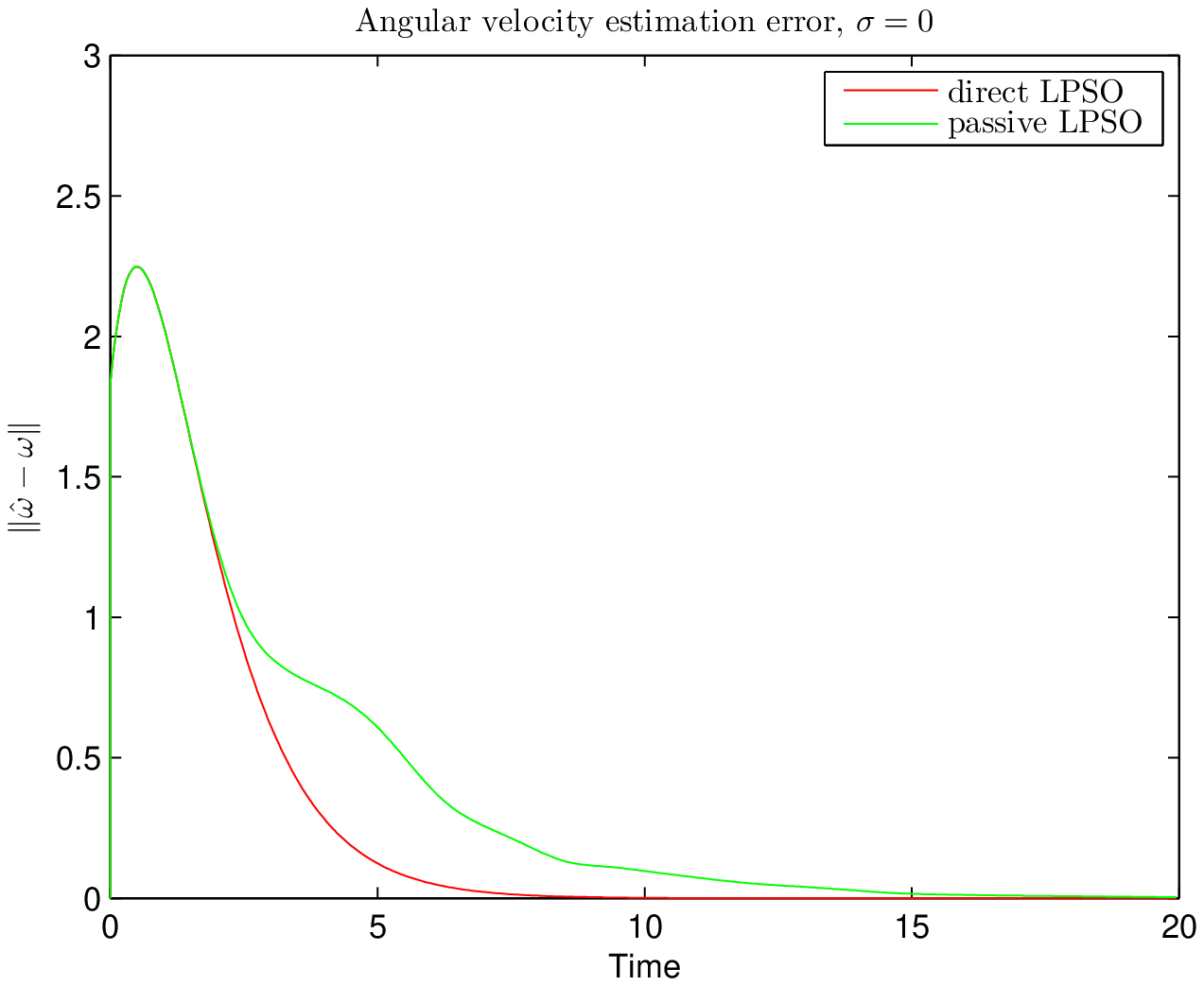}
\caption{$\|\hat{\omega} - \omega\|$ versus time with $\sigma = 0$.}
\label{fig:tiger}
\end{subfigure}

\begin{subfigure}{0.5\textwidth}
\centering
\includegraphics[width=\textwidth]{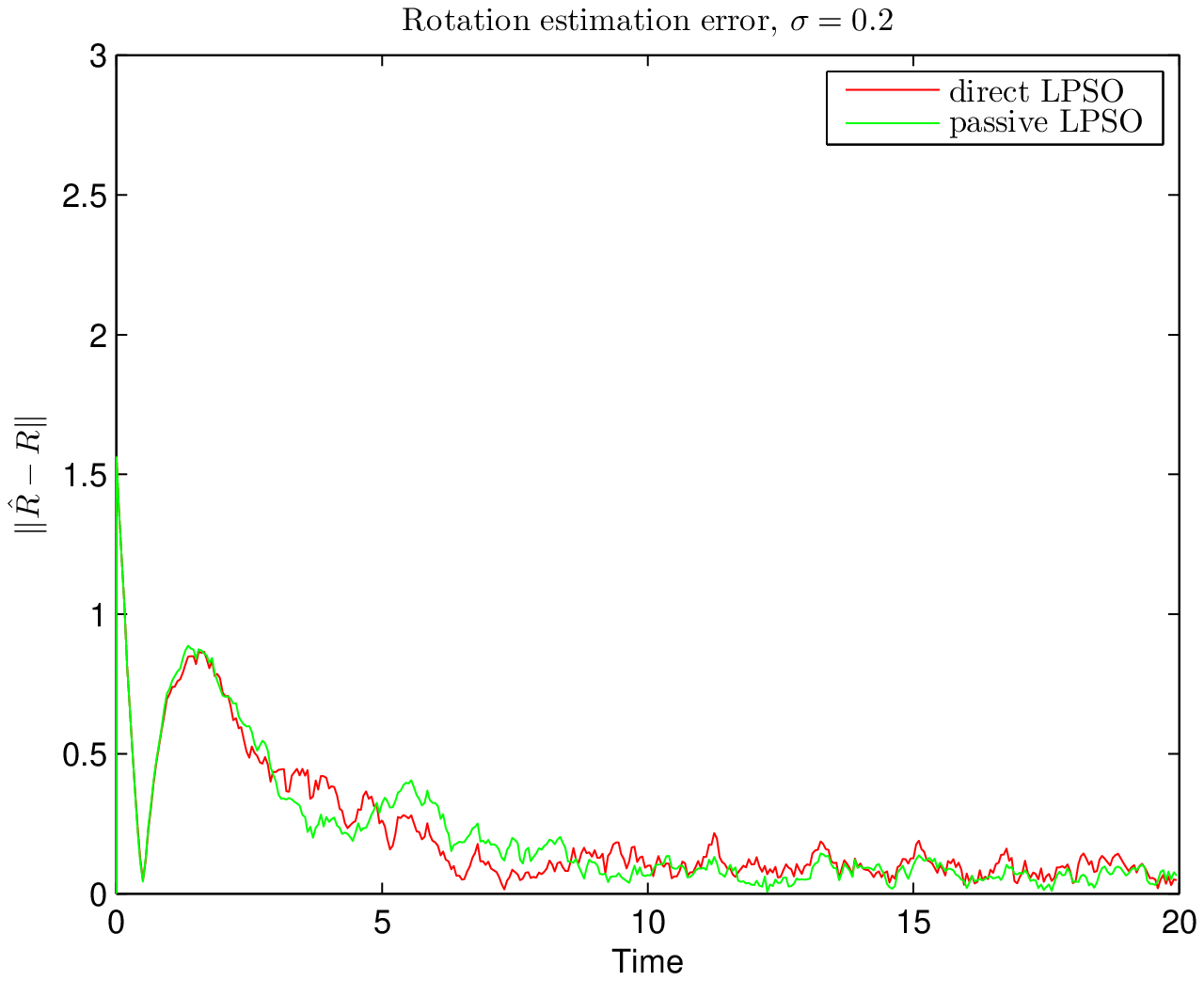}
\caption{$\|\hat{R} - R\|$ versus time with $\sigma = 0.2$.}
\label{fig:gull}
\end{subfigure}%
~
\begin{subfigure}{0.5\textwidth}
\centering
\includegraphics[width=\textwidth]{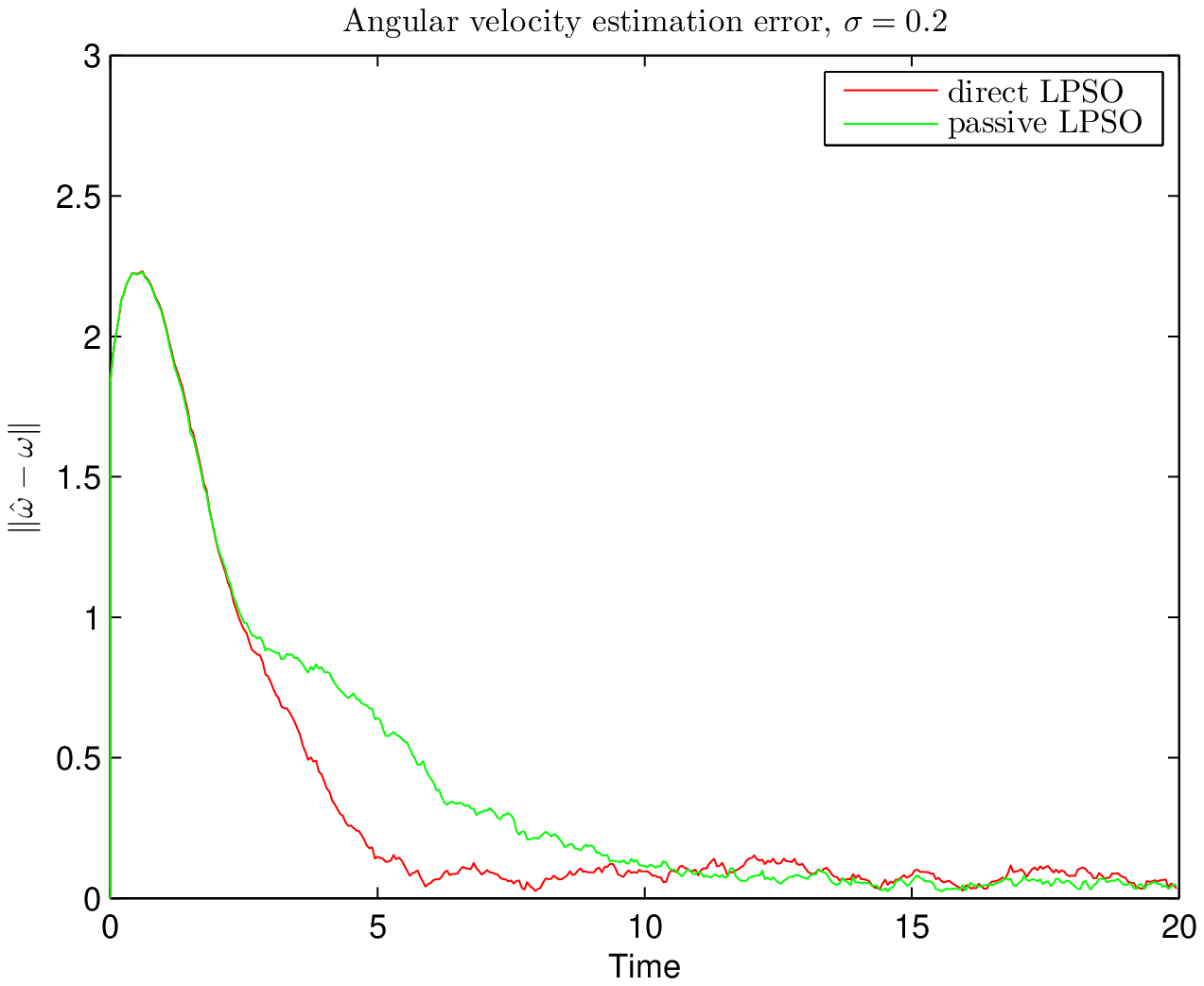}
\caption{$\|\hat{\omega} - \omega\|$ versus time with $\sigma = 0.2$.}
\label{fig:tiger}
\end{subfigure}

\begin{subfigure}{0.5\textwidth}
\centering
\includegraphics[width=\textwidth]{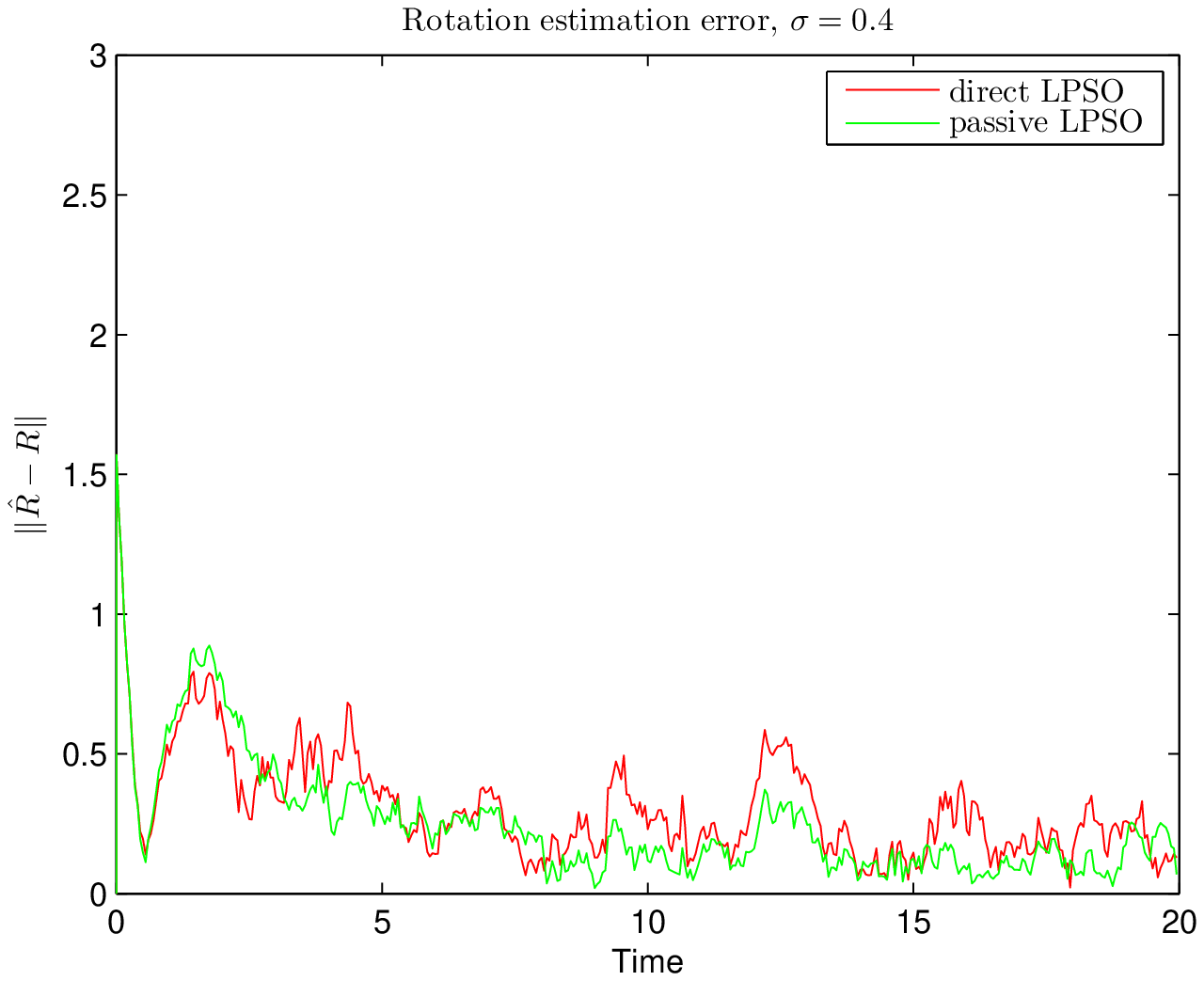}
\caption{$\|\hat{R} - R\|$ versus time with $\sigma = 0.4$.}
\label{fig:gull}
\end{subfigure}%
~
\begin{subfigure}{0.5\textwidth}
\centering
\includegraphics[width=\textwidth]{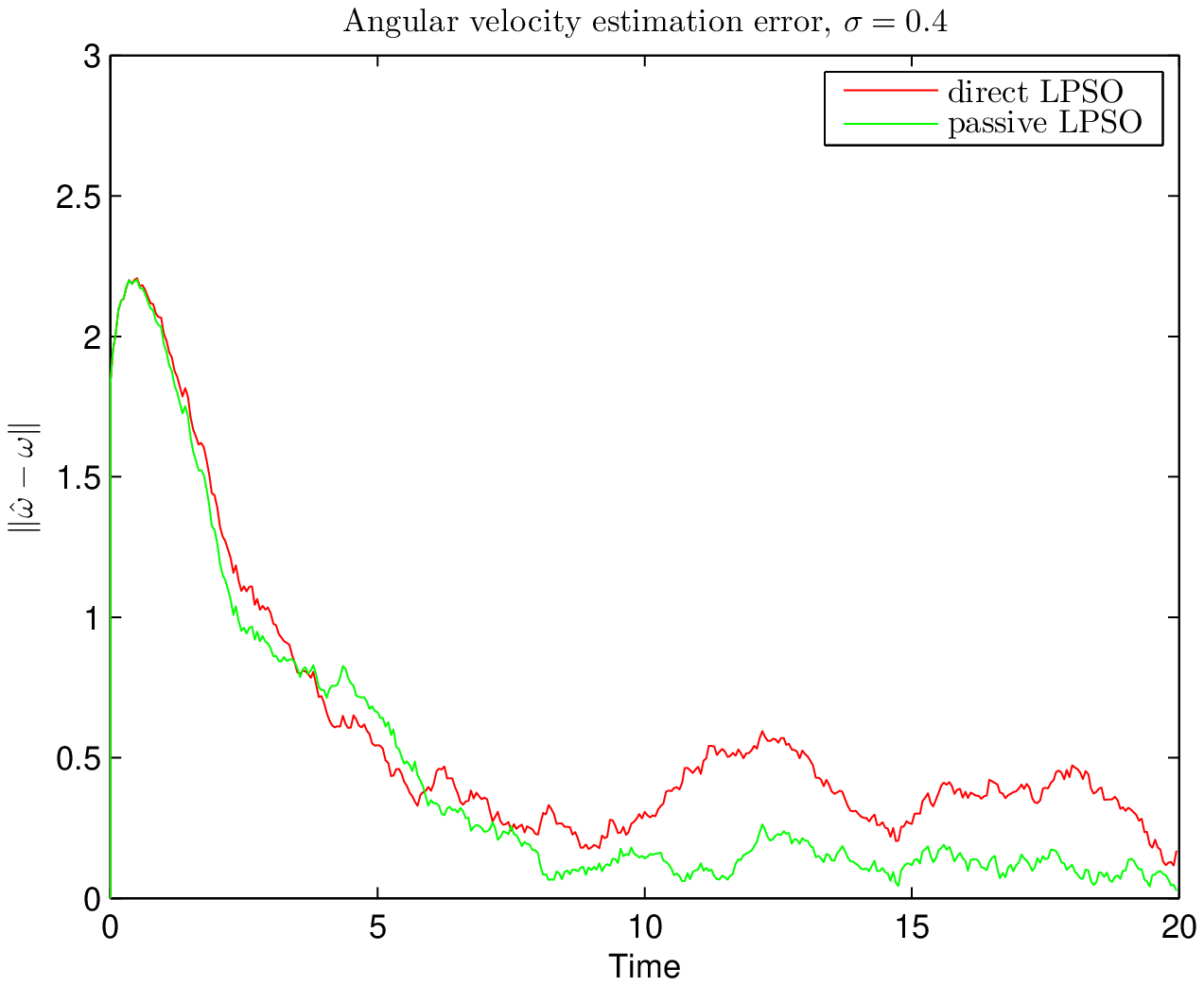}
\caption{$\|\hat{\omega} - \omega\|$ versus time with $\sigma = 0.4$.}
\label{fig:tiger}
\end{subfigure}
\caption{Direct and passive LPSOs for dynamic system on $\SO{3,
    \Real}$ with increasing amounts of measurement noise.}
\label{fig:so3_sim_05}
\end{figure}
\section{Conclusions}

We have proposed observers for two different classes of systems on
linear Lie groups. The first class of system is one in which the
entire state evolves on the general linear group and the entire state
is measured. We call observers for this class of system Lie group
full-state observers. We have shown that if the systems state is
bounded, then both the left and right invariant estimation errors are
differentially equivalent to a stable LTI system and hence are locally
exponentially stable. The second class of system is one in which only
part of the state evolves on the general linear group and only this
portion of the state is measured. We call observers for this class of
system Lie group partial-state observers. We have shown that if the
system's state is bounded, then the left and right estimation errors
are locally exponentially stable using the direct observer. For this
class of system the passive observer was shown to work well in
simulation. In all cases the observers were shown to work well in
simulation in the presence of constant disturbances.

\bibliographystyle{ieeetr}
\bibliography{LieObservers}




\end{document}